\documentclass[12pt]{amsart}
\usepackage{amsmath,jswmacros}

\DeclareMathOperator{\SC}{SC}
\DeclareMathOperator{\PS}{PS}
\DeclareMathOperator{\ST}{St}
\DeclareMathOperator{\SP}{St}
 \DeclareMathOperator{\vol}{vol}
\DeclareMathOperator{\Types}{Types}
 \DeclareMathOperator{\JL}{JL}
\numberwithin{equation}{section}

\begin{document}

\title{Hilbert modular forms with prescribed ramification}
\author{Jared Weinstein}
\email{jared@math.ucla.edu}
\address{UCLA Department of Mathematics,  Box 951555
Los Angeles, CA 90095-1555 USA}


\begin{abstract}
Let $K$ be a totally real field. We present an asymptotic formula
for the number of Hilbert modular cusp forms $f$ with given
ramification at every place $v$ of $K$. When $v$ is an infinite
place, this means specifying the weight of $f$ at $k$, and when $v$
is finite, this means specifying the restriction to inertia of the
local Weil-Deligne representation attached to $f$ at $v$. Our
formula shows that with essentially finitely many exceptions, the
cusp forms of $K$ exhibit every possible sort of ramification
behavior.
\end{abstract}

\maketitle

\vspace*{6pt}

\section{Introduction and Main Theorem}

Let $K$ be a totally real field.  Our investigation is concerned
with counting the number of cuspidal automorphic representations
$\pi=\otimes_v \pi_v$ of the adele group $\GL(2,\mathbf{A}_K)$ whose
local components $\pi_v$ have prescribed ramification for all places
$v$ of $K$.  We must explain what is meant by ``prescribed
ramification":  When $v$ is an infinite place, it means that $\pi_v$
is a prescribed essentially discrete series representation of $\GL(2,\R)$   When $v$ is a finite place, it means that the Weil-Deligne
representation associated to $\pi_v$ under the local Langlands
correspondence has prescribed restriction to inertia and monodromy operator.

Our problem may be restated in terms of the $\ell$-adic Galois representations attached to Hilbert modular eigenforms over $K$.  Suppose $f$ is such a form with coefficients in $\overline{\Q}_\ell$.   Then there is a corresponding $\ell$-adic Galois representation $\rho_f\from \Gal(\overline{K}/K)\to\GL(2,\overline{\Q}_\ell)$, as in~\cite{TaylorGaloisReps}.  Fix an $\ell$ and suppose the following data are given:

\begin{enumerate}
\item For each infinite place $v$, a weight $k_v\geq 2$, and
\item For each finite place $v$ not dividing $\ell$, a representation $\sigma_v\from I_{K_v}\to \GL(2,\overline{\Q}_\ell)$ of the inertia group $I_{K_v}$ which extends to the full Galois group $\Gal(\overline{K}_v/K_v)$.  Assume almost all of these are trivial.
\end{enumerate}

We are then concerned with counting the number of Hilbert modular forms $f$ of weights $(k_v)$ and level prime to $\ell$ for which the restriction of $\rho_f$ to $I_{K_v}$ is $\sigma_v$ for each finite place $v$.   Our main Theorem~\ref{maintheorem} gives an asymptotic formula for the number of such forms.  Barring a natural obstruction coming from the central character, it shows that such forms always exist unless the data $(k_v)$ and $(\sigma_v)$ given above belong to a finite set of exceptional data up to twisting.

This theorem affirms the existence, at least for $\GL(2)$, of automorphic representations subject to local constrains which are more stringent than those previously considered in the literature. Clozel~\cite{ClozelLimitMultiplicities}, building on a result of DeGeorge and Wallach~\cite{DeGeorgeWallach}, begins with a real semisimple Lie group $G$ and a discrete series representation $\delta$ of $G$ and then uses a trace formula to count the multiplicity of $\delta$ in $L^2(\Gamma\backslash G)$ as $\Gamma$ ranges through a tower of arithmetic subgroups of $G$:  This is akin to counting the number of modular forms of a given weight $k$ whose level is supported on a finite set of primes, though it may be deeply ramified at those primes.   Chenevier~\cite{Chenevier} uses the Peter-Weyl theorem to construct automorphic representations of a certain sort of unitary group (which is compact at the infinite places) which has prescribed ramification at a prime but arbitrary behavior at infinity;  this representation is then used to construct number fields with given ramification.  Our present investigation is limited to the group $\GL(2)$, but we control both the weight and the ramification at all finite places.  In this sense our main theorem is similar to a theorem of Khare and Prasad~\cite{khareprasad}, which shows the existence of classical cuspidal eigenforms of weight 2 for the principal congruence subgroup $\Gamma(p)$ which have a particular sort of ramification behavior at one prime $p$.  In Section~\ref{Q} we offer some very detailed information in the case of $K=\Q$;  in particular, we compute the class of the space of cusp forms $S_k(\Gamma(N))$ in the Grothendieck group of $\SL(2,\Z/N\Z)$, thus generalizing the classical dimension formulas in~\cite{Shimura:Automorphic} and~\cite{CohenOesterle}.

To state our main theorem, we need to introduce some notation regarding ``inertial types" for $GL(2)$ over a local or global field.

\subsection{Local Inertial Types.}  Let $F$ be a finite extension of $\Q_p$, with ring of integers
$\OO_F$. Let $\mathcal{A}_2(F)$ be the set of isomorphism classes of
complex-valued irreducible admissible representations of $\GL(2,F)$.
By the local Langlands correspondence, there is a bijection
$\pi\mapsto \sigma(\pi)$ mapping $\mathcal{A}_2(F)$ onto the set of
isomorphism classes of two-dimensional Frobenius-semisimple Weil-Deligne representations
of $F$ preserving $L$-functions and epsilon factors, see for instance~\cite{Kutzko:1}.  (See~\cite{tate:ntb} for the definition of Weil-Deligne representation and for the construction of its $L$-function and epsilon factor.)

Let $\OO_F$ be the ring of integers of $F$, so that $\GL(2,\OO_F)$ is a maximal compact open subgroup of $\GL(2,F)$.  In~\cite{Henniart:types} it
is shown that if $\pi\in\mathcal{A}_2(F)$, then
$\pi\vert_{\GL(2,\OO_F)}$ contains an irreducible finite-dimensional
subspace $\tau(\pi)$ of $\GL(2,\OO_F)$ which characterizes the restriction of
$\sigma(\pi)$ to the inertia group of $F$.  We shall call $\tau(\pi)$ the inertial type of $\pi$.
We will leave the precise definition of $\tau(\pi)$ for Section~\ref{types}, but we remark that when
$\pi$ belongs to the unramified principal series, $\tau(\pi)$ is the
trivial representation. Let $\Types(F)$ be the set of isomorphism
classes of representations of $\GL(2,\OO_F)$ which arise as inertial types
for members of $\mathcal{A}_2(F)$.  For $\tau\in\Types(F)$ we define
the quantity
$$d(\tau)=\begin{cases}
q-1,&\text{ $\tau$ is the type of a special representation,}\\
\dim \tau,&\text{ otherwise;}\end{cases}$$  here $q$ is the
cardinality of the residue field of $\OO_F$.   (A special representation of $\GL(2,F)$ is a twist of the Steinberg representation of this group.  Weil-Deligne representations corresponding to special representations are precisely those with nontrivial monodromy operator.)

Now suppose $F=\R$.  When $k\geq 2$ and $w$ are two integers of the same parity, let $\mathcal{D}_{k,w}$ be the essentially discrete series representation of $\GL(2,\R)$ as in 0.2 of~\cite{Carayol:ladicreps2}.  Then the central character of $\mathcal{D}_{k,w}$ is $t\mapsto t^{-w}$.  Let $\Types(F)$ denote the set of all such representations $\mathcal{D}_{k,w}$.  If $\pi$ is such a representation we simply define $\tau(\pi)=\pi$.  We define the function $d$ on $\Types(F)$ by $d(\mathcal{D}_{k,w})=k-1$.

In either case, suppose $\tau\in\Types(F)$.  When $\chi$ is a (one-dimensional) character of $F^*$ we denote by $\tau\otimes\chi$ the representation $g\mapsto \chi(\det g)\tau(g)$;  this also belongs to $\Types(F)$.

\subsection{Global Inertial Types.}  Now suppose $K$ is a totally real field of degree $n$.
Let $S$ (resp., $S_f$, $S_\infty$) be the set of places (resp., finite places, infinite places) of $K$.  To a cuspidal automorphic
representation $\pi$ of $\GL(2)_K$ arising from a Hilbert modular
form, we can associate the representation $\tau(\pi)=\otimes_{v\in S}\tau(\pi_v)$ of $\GL(2,\hat{\OO}_K\times (K\otimes\R))$.  Loosely speaking, $\tau(\pi)$  measures the ramification of $\pi$ at the finite places and  records the components of $\pi$ at the infinite places.  If the collection of infinite places is denoted $\set{v_1,\dots, v_n}$ and if $\pi_{v_i}\isom \mathcal{D}_{k_i,w_i}$, then  $w_1=\dots=w_n$ and the integers $k_i$ and $w_i$ all have the same parity.  Such a representation $\pi$ arises from a Hilbert modular form of weights $(k_1,\dots,k_d)$, {\em cf}.~\cite{Ohta}

We have that $\tau(\pi_v)$ is the trivial representation for all finite places not dividing the level of $\pi$.  Furthermore, the central character $\chi$ of $\pi$ is an algebraic Hecke character of $\mathbf{A}_K^*$ whose restriction to $\OO_{K,v}^*$ (resp., $K_v^*$) equals the central character of $\tau(\pi_v)$ for all $v\in S_f$ (resp., $v\in S_\infty$).  In light of this we define the
set $\Types(K)$ of {\em global inertial types} to consist of the collections $\tau=(\tau_v)_{v\in
S}$ satisfying the conditions:

\begin{enumerate}
\item For all but finitely many $v$, $\tau_v$ is the trivial
representation.

\item There exists an algebraic Hecke character of $\mathbf{A}_K$ whose component at each
place $v$ agrees with the central character of $\tau_v$.  (If one
exists, then there are as many as the class number of $F$.)
\end{enumerate}

We remark that condition (2) is equivalent to the condition that if
$\chi_v$ is the central character of $\tau_v$, then
\begin{equation}
\label{condition2} \chi=\prod_v\chi_v\text{ vanishes on $\OO_K^*$.}
\end{equation}
Indeed, $\chi$ is a character of the subgroup $\prod_{v\in S_f }\OO_{K_v}^*\times\prod_{v\in S_\infty}K_v^*$ of $\mathbf{A}_K^*$, and for $\chi$ to extend to an algebraic Hecke character it is necessary and sufficient that $\chi$ vanish on the intersection of this subgroup with the diagonally embedded group $K^*\subset\mathbf{A}_K^*$.  This intersection is exactly the unit group $\OO_K^*$.

If $\tau\in\Types(K)$ is a global inertial type, we shall write $\tau=\tau_f\otimes\tau_\infty$ to denote the decomposition of $\tau$ into its finite and infinite components.  Note that $\tau_f$ is a finite-dimensional representation of the compact group $\GL(2,\hat{\OO}_F)$ and $\tau_\infty$ is a representation of $\GL(2,K\otimes\R)$.  If we write $\tau_{v_i}\isom D_{k_i,w_i}$ for the infinite places $v_i$, then Eq.~\ref{condition2} and the Dirichlet Unit Theorem imply that $w_1=\dots=w_n$.

Whenever $\pi$ is a cuspidal automorphic representation of
$\GL(2)_K$ arising from a Hilbert modular form, $\tau(\pi)$ belongs
to $\Types(K)$.  For any $\tau\in\Types(K)$, we define
$$d(\tau)=\prod_v d(\tau_v),$$ the product making sense because all but finitely many factors are 1.

The notion of a global type $\tau$ being a twist of another type $\tau'$ is evident:  this shall mean that for each finite (resp. infinite) place $v$ there exist characters $\chi_v$ of $\OO_v^*$ (resp. $K_v^*$) such that $\tau_v=\tau_v'\otimes\chi_v$ for all places.  (This can only be so if $\prod_v\chi_v^2$ vanishes on $\OO_K^*$.)

\subsection{Main Theorem}

Let $S(\tau)$ denote the set of Hilbert modular forms $\pi$ for
which $\tau(\pi)=\tau$. Our main theorem is an estimate for the
cardinality of $S(\tau)$.

\begin{Theorem}
\label{maintheorem} We have $$\#
S(\tau)=2^{1-n}\abs{\zeta_K(-1)}hd(\tau)+O\left(2^{\nu(\tau)}\right),$$
where
\begin{eqnarray*}
\zeta_K(s)&=&\text{ the Dedekind zeta function for $K$}\\
h&=&\text{ the class number of $K$}\\
\nu(\tau)&=&\text{ the number of finite places $v$ for which
$\dim \tau_v>1$. }
\end{eqnarray*}
The constant in the ``$O$" only depends on the field $K$.
\end{Theorem}

Note that $\zeta_K(-1)$ is a nonzero (in fact rational) number.  By comparing the quantity $d(\tau)$ with the error term $2^{\nu(\tau)}$ (see Section~\ref{typesdef}), we will deduce the following

\begin{Cor}  \label{maincor} Up to twisting by one-dimensional characters, the set of global inertial types $\tau\in\Types(K)$ for which $S(\tau)=\emptyset$ is finite.
\end{Cor}

Stated rather loosely, this means that there always exists a Hilbert
modular form with prescribed ramification data, so long as:  the
desired weight is large enough at one of the infinite places, OR the
desired inertial representation at finite place $v$ is ramified
deeply enough, OR enough primes are permitted to ramify.  See
Section~\ref{Q} for an explicit account of the case $K=\Q$.



We have found it most convenient to divide the proof of
Theorem~\ref{maintheorem} into two cases, depending on the parity of
$n=[K:\Q]$.  If $n$ is even, we work with the definite quaternion
algebra $D$ ramified exactly at the infinite places. If $n$ is odd,
we work with the quaternion algebra $D$ ramified at all but one of
the infinite places.  In each case, we wish to compute the
multiplicity of a global type in a space of automorphic forms on a
Shimura variety corresponding to $D$.  The dimension of the Shimura
variety will be 0 or 1 as $n$ is even or odd, respectively.

\section{Types for $\GL(2)$}
\label{types}

\subsection{Definition of Types}
\label{typesdef}
Let $F$ be a $p$-adic field with ring of integers $\OO_F$, maximal
ideal $\gp_F$, and residue field $\mathbf{F}_q$.  In this section we
gather the necessary definitions and facts concerning types for
$\GL(2,F)$.

We must first define the association $\pi\mapsto\tau(\pi)$ attaching
types to objects of $\mathcal{A}_2(F)$.  We do this by cases as
follows.  It will simplify matters to assume in the following
discussion that $\pi$ is of minimal conductor among its twists by
characters of $F^*$.   Types for $\pi$ not satisfying this condition
can be defined via the relation
$\tau(\pi\otimes\chi)=\tau(\pi)\otimes\chi\vert_{\OO_F^*}$.

\begin{enumerate}
\item
If $\pi$ belongs to the unramified principal series then $\tau(\pi)$ is the trivial representation of $\GL(2,\OO_F)$.
\item If $\pi$ is the principal series representation corresponding to a pair of characters $\chi_1,\chi_2$ of $F^*$ with $\chi_1\chi_2^{-1}$ ramified, twist $\pi$ so as to assume that $\chi_2$ is trivial.  Let $\gp^{\mathfrak{c}}$ be the conductor of $\chi_1$.  Then $\tau(\pi)$ is induced from the character
$\tbt{a}{b}{c}{d}\mapsto\chi_1(a)$ of
$$\Gamma_0(\gp^{\mathfrak{c}}):=\set{\tbt{a}{b}{c}{d}\in\GL(2,\OO_F)\biggm{\vert} c\equiv
0\pmod{\gp^{\mathfrak{c}}}}.$$  By~\cite{Casselman}, Prop. 1, $\tau(\pi)$ is an irreducible representation of $\GL(2,\OO_F)$.  We then have $d(\pi)=\dim \tau(\pi)=q^{c-1}(q+1)$.
\item If $\pi=\St$ is the Steinberg representation of $\GL(2,F)$, then $\tau(\pi)$ is pulled back from the unique irreducible $q$-dimensional
representation $\St_{\GL(2,\OO_F)}$ of $\GL(2,\mathbf{F}_q)$
contained in the permutation representation on
$\mathbf{P}^1(\mathbf{F}_q)$.  In this case we have defined $d(\tau(\pi))=q-1$.
\item If $\pi$ is supercuspidal, then it is induced from a finite-dimensional irreducible representation $\lambda$ of a compact-mod-center subgroup $J\subset\GL(2,F)$.  These were constructed in~\cite{KutzkoSupercuspidal} and ~\cite{KutzkoSupercuspidal2};  we review the construction in Section~\ref{tracebounds}.
By replacing $J$ with a conjugate we may assume that the maximal compact subgroup
$J^0$ of $J$ lies in $\GL(2,\OO_F)$;  then
$\tau(\pi)=\Ind_{J^0}^{\GL(2,\OO_F)}\lambda\vert_{J^0}$ is irreducible;  see~\cite{Henniart:types}, A.3.1.    When $\sigma(\pi)$ is induced from a character $\theta$ of a quadratic extension $F'/F$ of conductor $c\geq 1$, we will see in the next section that the dimension of $\tau(\pi)$ is given by $$\dim\tau(\pi)=\begin{cases} q^{c-1}(q-1),&\text{ $F'/F$ unramified,}\\ q^{c-2}(q^2-1),&\text{ $F'/F$
 ramified.}   \end{cases}$$
\end{enumerate}

We shall call a type $\tau(\pi)$ principal series (resp., special, supercuspidal) when $\pi$ is principal series (resp., special, supercuspidal). From~\cite{Henniart:types}, A.1.5 we deduce the theorem:

\begin{Theorem}
\label{typetheorem} Let $\pi,\pi'\in\mathcal{A}_2(F)$.  The
following are equivalent:
\begin{enumerate}
\item $\pi'\vert_{\GL(2,\OO_F)}$ contains $\tau(\pi)$.
\item $\sigma(\pi)\vert_{I_F}\isom\sigma(\pi)\vert'_{I_F}$ {\em or} $\pi=\chi\otimes\St$ and $\pi'$ is
the principal series representation attached to two unequal
characters of $F^*$ whose restriction to $\OO_F^*$ agrees with
$\chi\vert_{\OO_F^*}$.
\end{enumerate}
\end{Theorem}

Therefore if $\pi,\pi'\in\mathcal{A}_2(F)$ and $\pi'$ contains
$\tau(\pi)$, then $\tau(\pi')=\tau(\pi)$, unless we are in the case
of the ``or" clause above, in which case $\tau(\pi)$ is a
one-dimensional character $\chi$ and
$\tau(\pi')=\chi\otimes\St_{\GL(2,\OO_F)}$.

A word on the dimensions of types is in order.  Suppose $\tau$ runs through a sequence of irreducible representations of $\GL(2,\OO_F)$.  If we assume that no member of this sequence is the twist of any other by a one-dimensional character, then we must have $\dim\tau\to\infty$.   Of course this implies $d(\tau)\to\infty$ as well.  Furthermore, an irreducible representation of $\GL(2,\OO_F)$ of least dimension other than 1 is one that is inflated from a cuspidal representation of $\GL(2,\mathbf{F}_q)$, and this has dimension $q-1$.  Therefore we have the lower bound $d(\tau)\geq q-1$ whenever $\tau$ is not one-dimensional.

We shall now deduce Cor.~\ref{maincor} from Thm.~\ref{maintheorem}.   Let $K$ be a totally real field and suppose that $\tau$ runs through a sequence of global inertial types of $K$, no two of which are twists of each other.   To prove the corollary it will suffice to show that the sequence $d(\tau)/2^{\nu(\tau)}$ tends to infinity.  Assume instead that it is bounded.  By the previous paragraph we have the lower bound
$$d(\tau)/2^{\nu(\tau)} \geq \prod_v \frac{q_v-1}{2},$$
where the product runs over the finite places $v$ for which $\tau_v$ is not one-dimensional.   Since the left hand side of this is bounded, the global types $\tau$ can only have higher-dimensional components at a finite set $S_0$ of finite places.  This implies that $\nu(\tau)$ is bounded from above;  by the hypothesis that $d(\tau)/2^{\nu(\tau)}$ is bounded we must have that $d(\tau)$ is bounded as well.  We claim there must exist a place $v_0\in S_0$ for which $\tau_{v_0}$ assumes infinitely many distinct values up to twisting.  The alternative is that for each place $v$ in $S_0$, the sequence $\tau_v$ comprises only finitely many types along with their twists;  by the pigeonhole principle, this would imply that the sequence $\tau$ contains two types which are twists of one another, contradicting the hypothesis.  But then by the previous paragraph $d(\tau_{v_0})\to \infty$, contradicting the boundedness of $d(\tau)$.   The corollary is proved.

\subsection{Trace bounds for types}\label{tracebounds}
Define the {\em level} of a type $\tau$ to be the least integer
$\ell$ for which $\tau$ factors through $\GL(2,\OO_F/\gp^\ell)$.  We
shall say that the {\em essential level} of $\tau$ is the least
level of all the twists of $\ell$.

We need a lemma concerning the non-archimedean local types.

\begin{lemma}
\label{tracebound} Let $g\in\GL(2,\OO_F)$ be a matrix whose characteristic polynomial has discriminant in $\OO_F^*$.   Then for all  $\tau\in\Types(F)$, we have
\begin{equation}
\label{traceineq}
\begin{cases}\abs{\tr(\tau(g))}=1,&\text{$\tau$ one-dimensional or special,}\\
\abs{\tr(\tau(g))}\leq 2,&\text{all other cases}.\end{cases}
\end{equation}
Furthermore, if we relax the hypothesis on $g$ and merely assume that $g$ has distinct eigenvalues in $\overline{F}^*$, then $\abs{\tr\tau(g)}$ is bounded as $\tau$ ranges through $\Types(F)$.
\end{lemma}

\begin{proof}  We proceed by taking cases with respect to the structure of $\tau$.  If $\tau$ is one-dimensional, the inequality of the lemma is obvious.  If $\tau$ is special, then up to twisting it is the inflation of the Steinberg representation $\St_{\GL(2,\mathbf{F}_q)}$.   If the characteristic polynomial of $g\in\GL(2,\OO_F)$ has unit discriminant then its reduction $\overline{g}\in\GL(2,\mathbf{F}_q)$ has distinct eigenvalues.  From the formula
$$\St_{\GL(2,\mathbf{F}_q)}=[\Ind_B^{\GL(2,\mathbf{F}_q)} 1]  - [1],\; B=\set{\tbt{a}{b}{0}{d}},$$
we see easily that $\tr\left(\St_{\GL(2,\mathbf{F}_q)}(\overline{g})\right)$ equals 1 if $\overline{g}$ has eigenvalues in $\mathbf{F}_q^*$ and $-1$ otherwise, thus establishing Eq.~\ref{traceineq}.    Of course we have $\abs{\tr\tau(g))}\leq\dim\tau=q$ no matter what $g$ is.

We now give a more explicit description of $\tau$ in the remaining cases.  In each of these cases $\tau$ is induced from a ``small" representation of a subgroup of $\GL(2,\OO_F)$.  Since we are interested in the trace of $\tau$, the following form of Mackey's theorem will be useful:
\begin{equation}
\label{mackey}
\tr\left(\Ind_H^G\eta\right)(g)=\sum_{[x]\in\left(G/H\right)^g}\tr\eta(x^{-1}gx),
\end{equation}
where the sum ranges over cosets $[x]=xH\in G/H$ which are fixed under the left action of $g$.  We will be applying this equation to various finite-index subgroups $J^0\subset \GL(2,\OO_F)$.

Suppose $\tau$ is the type of a ramified principal series representation.   By replacing $\tau$ with a twist we may assume, as in the previous subsection, that $\tau$ is induced from the character $\tbt{a}{b}{c}{d}\mapsto\chi(a)$ of $\Gamma_0(\gp^{\mathfrak{c}})$.    Note that we can identify the quotient $\GL(2,\OO_F)/\Gamma_0(\gp^{\mathfrak{c}})$ with the projective line $\mathbf{P}^1(\OO_F/\gp^{\mathfrak{c}})$, together with its natural action of $\GL(2,\OO_F)$.   If $\overline{g}$ has distinct eigenvalues then $g$ only has zero or two fixed points on $\mathbf{P}^1(\OO_F/\gp^{\mathfrak{c}})$ and therefore $\abs{\tr\tau(g)}\leq 2$ in this case.  If we lift this assumption and merely assume that $g$ has distinct eigenvalues in $\overline{F}^*$, then $g$ may have more than two fixed points on $\mathbf{P}^1(\OO_F/\gp^{\mathfrak{c}})$.  However, we claim that as$\mathfrak{c}\to\infty$ then this number of fixed points remains bounded.  To establish this, we may pass from $F$ to an extension $F'$ containing the eigenvalues of $g$ and then show that $g$ has a bounded number of fixed points on $\mathbf{P}^1(\OO_{F'}/\gp_{F'}^{\mathfrak{c}})$ as $c\to\infty$.  Working over $F'$, we have up to conjugacy $g=\tbt{\alpha}{}{}{\beta}$, with $\alpha,\beta\in\OO_{F'}^*$;  let $m$ be the valuation in $F'$ of $\alpha-\beta$.  Then for $\mathfrak{c}\geq m$, the fixed points of $g$ on $\mathbf{P}^1(\OO_{F'}/\gp_{F'}^{\mathfrak{c}})$ are exactly the points $[x:1]$ and $[1:y]$, where $x$ and $y$ range through those elements of $\OO_{F'}/\gp^{\mathfrak{c}}$ of valuation at least $c-m$.  Thus $g$ has at most $2(\#\OO_{F'}/\gp_{F'})^m$ fixed points which are $F$-rational, and this number is bounded as $\mathfrak{c}\to\infty$, proving the claim.  Therefore $\abs{\tr\tau(g)}$ is bounded as $\tau$ ranges through all types of ramified principal series representations.

Now suppose $\tau$ is the type of a minimal supercuspidal representation $\pi$.  To prove the required inequality, we need to give a more detailed description of $\tau$.  Here we follow~\cite{Henniart:types}, A.3.2--A.3.8, wherein an exhaustive list of the supercuspidal types is given.  Let $k=\OO_F/\gp_F$, let $\varpi_F$ be a uniformizer of $F$, and let $\psi$ be an additive character of $F$ which vanishes on $\gp_F$ but not on $\OO_F$.  There are three cases to consider:

\begin{enumerate}
\item {\em $\pi$ has conductor $\gp^2$.}In this case $\tau$ is inflated from an irreducible cuspidal representation of $\GL_2(k)$.  These in turn are in correspondence with certain characters $\theta$ of the multiplicative group of a quadratic extension $k_2/k$.  If $\theta$ is a character of $k_2^*$ unequal to its $k$-conjugate, then there is a unique $(q-1)$-dimensional representation
$\tau_\theta$ of $\GL(2,k)$ satisfying
$$\tr\tau_\theta(\overline{g})=\begin{cases}  -(\theta(\alpha_1)+\theta(\alpha_2)),&\text{$\overline{g}$ has eigenvalues $\alpha_1,\alpha_2\in k_2^*\backslash k^*$}\\
0,&\text{$\overline{g}$ has distinct eigenvalues in $k^*$.}\end{cases}$$
Then the type $\tau$ is inflated from one of the representations $\tau_\theta$.  The desired inequality $\abs{\tr\tau(g)}\leq 2$ is obvious from the above description so long as $\overline{g}$ has distinct eigenvalues.

\item {\em $\pi$ has conductor $\gp_F^{\mathfrak{c}}$, $\mathfrak{c}\geq 4$ even.}  Let $E/F$ be the unique unramified quadratic extension field, and let $n=\mathfrak{c}/2$.  Let $b\in E^*$ be of the form $\varpi^{-n}u$, where  $u\in\OO_E^*$ has residue class in $\OO_E/\gp_E$ which generates that field over $k$.  Let $\theta$ be a character of $E^*$ for which $\theta(1+x)=\psi\circ\tr_{E/F}(bx)$ for $b\in\gp_E^{\floor{(n+1)/2}}$.

Choose an embedding $E\injects M_2(F)$ so that $M_2(\OO_F)\cap E = \OO_E$.  Define the subgroup $J^0\subset\GL(2,\OO_F)$ by $$J^0=\OO_E^*\left(1+\gp_F^{\floor{(n+1)/2}}M_2(\OO_F)\right).$$
Then $\tau=\Ind_{J^0}^{\GL(2,\OO_F)}\eta$ for a representation $\eta$ of $J^0$ which we now describe.  If $n$ is odd then $\eta$ is a character of $J^0$ defined by the conditions $\eta\vert_{\OO_E^*}=\theta\vert_{\OO_E^*}$ and $\theta(1+x)=\psi\circ\tr(bx)$ for $x\in\gp_F^{(n+1)/2}M_2(\OO_F)$.  If $n$ is even then $\eta$ is the unique irreducible representation of $J^0$ of dimension $q$ satisfying $\tr\eta(\alpha(1+x))=-\theta(\alpha)$ whenever $x\in\gp_F^{(n+1)/2}$ and $\alpha\in\OO_E^*$ is such that the image of $\alpha$ in $\OO_E/\gp_E$ does not lie in $k$.  We remark that representations $\pi$ of $\GL(2,F)$ containing the type $\tau$ have conductor $\gp_F^{2n}=\gp_F^{\mathfrak{c}}$.

To apply Mackey's theorem we need to consider the coset space
$\GL(2,\OO_F)/J^0$ together with its left $\GL(2,\OO_F)$-action.  If
we let $$\mathcal{H}_n=\set{\alpha\in
\left(\OO_{E}/\gp_{E}^{\floor{(n+1)/2}}\right)^*\biggm{\vert} \alpha\mod\gp_E\not\in k},$$
then there is a natural left action $(g,b)\mapsto g\cdot b$ of
$\GL(2,\OO_F)$ on $\mathcal{H}_c$ via fractional linear
transformations.  If $b\in\mathcal{H}_n$ is a fixed point of $\OO_E^*$ then $g\mapsto g\cdot b$ gives a $\GL(2,\OO_F)$-equivariant bijection
\begin{equation}\label{GmodJ}
\GL(2,\OO_F)/J^0\isom\mathcal{H}_n,\end{equation} much as the
coset space $\GL(2,\R)/O(2,\R)$ is identified with (two copies of)
the upper half plane.  If $g\in\GL(2,\OO_F)$ is such that
$\overline{g}$ has distinct eigenvalues, then $g$ has at most two
fixed points in $\mathcal{H}_n$ and we find $\abs{\tr\tau(g)}\leq
2$.  (In fact if $g$ has eigenvalues $\alpha,\beta$ lying in
$\OO_{F'}$ then
$\tr\tau(g)=(-1)^{n-1}(\theta(\alpha)+\theta(\beta))$.)  On the other
hand if $g$ is merely assumed to have distinct eigenvalues then $g$
may have more than two fixed points on $\mathcal{H}_n$.  However,
the number of fixed points is bounded as $n\to\infty$, by the same argument given in the previous paragraph.  Therefore
$\abs{\tr\tau(g)}$ is bounded as $\tau$ runs through types of
supercuspidal representations of this sort.

\item {\em $\pi$ has conductor $\gp_F^{\mathfrak{c}}$, $\mathfrak{c}\geq 3$ odd.}
  Let $E/F$ be a ramified quadratic extension, and let $n=\mathfrak{c}-2$.  Let $b\in E^*$ have valuation $-n$.  Let $\theta$ be a character of $E^*$ satisfying $\theta(1+x)=\psi\circ\tr_{E/F}(bx)$ whenever $x\in\gp_E^{(n+1)/2}$.

Let $\mathfrak{A}\subset M_2(\OO_F)$ be the algebra $$\mathfrak{A}=\set{\tbt{a}{b}{c}{d}\biggm{\vert} c\in\gp_F},$$
and choose an embedding $E\injects M_2(F)$ in such a way that $\mathfrak{A}\cap E = \OO_E^*$.  Let $P_{\mathfrak{A}}\subset \mathfrak{A}$ be the double-sided ideal of matrices $\tbt{a}{b}{c}{d}$ with $a,c,d\in\gp_F$.
Our subgroup  $J^0$ is then
$$J^0=\OO_E^*\left(1+P_{\mathfrak{A}}^{(n+1)/2}
\right)$$
and $\eta$ is
the character
$\alpha(1+x)\mapsto\theta(\alpha)\psi(\tr(bx))
$ for $\alpha\in\OO_E^*$, $x\in P_{\mathfrak{A}}^{(n+1)/2}$.
Then $\tau=\Ind_{J^0}^{\GL(2,\OO_F)}\eta$ is a type contained in supercuspidal representations $\pi$ of $\GL(2,F)$ of conductor $n+2$.

Now suppose $g\in\GL(2,\OO_F)$ is such that $\overline{g}$ has distinct eigenvalues.  If $\overline{g}$ has irreducible characteristic polynomial, then no conjugate of $g$ can possibly lie in $\mathcal{A}^*$, let alone in $J^0$, so that $\tr\tau(g)=0$.  The alternative is that up to conjugacy $g\in\mathfrak{A}^*$ equals the diagonal matrix with eigenvalues $\alpha,\beta\in\OO_F^*$ whose residue classes are unequal.
Let $\mathcal{H}_n$ be the quotient of the set $\gp^{-1}\backslash\OO_E$ by the group $1+\gp_E^n$.
The analogue of Eq.~\ref{GmodJ} is the $\mathfrak{A}^*$-equivariant bijection $\mathfrak{A}^*/J^0\isom \mathcal{H}_n$.  Since $g$ has no fixed points on $\mathcal{H}_n$ we have $\tr\tau(g)=0$ as well.

Now assume only that $g$ has distinct eigenvalues.  Let $\lambda=\Ind_{J^0}^{\mathfrak{A^*}}\eta$.  Then by Mackey's Theorem $\tr\tau(g)$ is a sum over at most $\#\GL(2,\OO_F)/\mathfrak{A}^*=q+1$ terms of the trace $\tr\lambda$ evaluated on conjugates of $g$.  The same argument from the previous paragraph shows that an element $h\in\mathfrak{A}^*$ with distinct eigenvalues has a bounded number of fixed points on $\mathcal{H}_n$ as $n\to\infty$.  Thus $\abs{\tr\tau(g)}$ is bounded as $\tau$ runs through types of supercuspidal representations of this sort as well.
\end{enumerate}\end{proof}

Lemma~\ref{tracebound} has a global consequence which we will need in the sequel.  Let $K$ be a totally real field and suppose $B/K$ is a quaternion algebra (possibly $M_2(K)$) which is split at all finite places. Suppose $\OO_B\subset B$ is a maximal order.  Then for all finite places $v$ we may identify $\OO_B\otimes_{\OO_F} \OO_{F_v}$ with $\GL(2,\OO_{F_v})$.  Let $g_v$ be the image of $g\otimes 1$ under this isomorphism.  If $\tau=\tau_f\otimes\tau_\infty\in\Types(K)$, let $\tau_f'$ be the representation of $\OO_B^*$ defined by $g\mapsto\prod_{v\text{ finite}} \tau_v(g_v)$.
\begin{lemma}
\label{tracebound2} Let $g\in \OO_B$ be an element whose (reduced) characteristic polynomial has distinct eigenvalues in $\overline{K}^*$.
There is a constant $C$ depending only on $g$ (and of course $K$)
such that for all types $\tau=\tau_f\otimes\tau_\infty\in\Types(K)$,
$\abs{\tr{\tau_f(g)}}\leq C2^{\nu(\tau)-n_{\text{sp}}(\tau)}$, where
$n_{\text{sp}}(\tau)$ is the number of finite places of $K$ at which
$\tau$ is special.
\end{lemma}

\begin{proof} Indeed, if such a $g$ is given then only finitely many
finite places of $K$ will divide the discriminant of the characteristic polynomial of $g$.   Let $S_g$ be the set of such places.   For each $v\in S_g$, the preceding lemma shows that there exists a bound $C_v$ so that for all $\tau_v\in\Types(K_v)$ we have $\abs{\tr\tau_v(g)}\leq C_v$.  Let $C=\prod_{v\in S_g} C_v$.  For every finite place $v\not\in S_g$, Lemma~\ref{traceineq} shows that $\abs{\tr\tau_v(g_v)}\leq 1$ if $\tau_v$ is special or one-dimensional and $\abs{\tr\tau_v(g_v)}\leq 2$ otherwise.  Therefore if $\tau_f$ is the finite part of a global type, we have the inequality $$\abs{\tr\tau_f'(g)}=\prod_{v\text{ finite}} \abs{\tau_v(g_v)} \leq \left(\prod_{v\in S_g} C_v\right) 2^{\nu(\tau)-n_{\text{sp}}(\tau)}$$ as required.

\end{proof}

\section{Proof of Theorem~\ref{maintheorem} in the case of $[K:\Q]$ even}

Assume that $n=[K:\Q]$ is even.  Let $S_f$ and $S_\infty$ denote the
finite and infinite places of $K$, respectively.  Let $\mathbf{A}$
and $\mathbf{A}_f$ denote the adeles and finite adeles of $K$.
Finally, let $D/K$ be the quaternion algebra ramified exactly at
$S_\infty$, and let $G=D^*$ be the inner form of $\GL(2)$ corresponding
to $D$.  The Jacquet-Langlands correspondence $\JL\from \pi'\mapsto
\pi$ puts automorphic representations of $G(\mathbf{A})$ in bijection with those
automorphic representations of $\GL(2,\mathbf{A})$ which are discrete series at
the infinite places, see~\cite{JL}  We will also use the symbol $\JL$ to mean the local Jacquet-Langlands correspondence between representations of $D_v$ and those of $\GL(2,K_v)$ for any particular $v$.

Now suppose $\tau\in\Types(K)$ is a global type for $K$.  Choose a
maximal order $\OO$ of $D$.  For each finite place $v\in S_0$ there
exists an isomorphism $\OO_v\isom GL(2,\OO_{K_v})$, unique up to
conjugacy. Let $\tau'_v$ be the pull-back of $\tau_v$ through any
such isomorphism; this is well-defined up to isomorphism.  Let
$\tau_f'=\bigotimes_{v\in S_f}\tau_v'$;  this is a
 finite-dimensional irreducible representation of $\hat{\OO}^*$.  Let also
 $\tau'_{\infty}=\bigotimes_{v\in S_\infty} \JL^{-1}(\tau_v)$;  this
 is a finite-dimensional representation of $G(K\otimes\R)=\prod_{v\in
 S_\infty} G(K_v)$.  Finally, let $\tau'=\tau'_f\otimes\tau'_\infty$,  this is a
 representation of $G(\hat{\OO}_K\times (K\otimes\R))$.
 We record the relationship
\begin{equation}
\label{dvsdim}
 d(\tau)=\dim\tau'\prod_{v\in
 S_{\text{sp}}}\left(1-\frac{1}{q_v}\right),
\end{equation}
where $S_{\text{sp}}$ is the set of places at which $\tau$ is
special and $q_v$ is the cardinality of the residue field of a
finite place $v$.

Because multiplicity one holds for $\GL(2)$ and for $G$, we see that counting the
automorphic representations of $GL(2,\mathbf{A}_K)$ whose restriction to
$\GL(2,\hat{\OO}_K\times(K\otimes\R))$ contains $\tau$ is the same as counting the automorphic representations of $G(\mathbf{A}_K)$ whose restriction
to $G(\hat{\OO}_K\times(K\otimes\R))$ contains $\tau'$. Write
$\mu(\tau)$ for the number of automorphic representations .  It is not necessarily the
case that $\mu(\tau)$ is the cardinality of $S(\tau)$, the set of
automorphic representations of type $\tau$, due to the possibility
of special components.  We will compute $\#S(\tau)$ in terms of
$\mu(\tau)$ at the end of the section.

To compute $\mu(\tau)$, we first realize it as the multiplicity of
$\tau'_f$ inside a space of automorphic forms of ``weight
$\tau'_\infty$", namely
\begin{equation}
\label{modforms1}
M(\tau_\infty')=\mathcal{C}_{G(K\otimes\R)}\left(G(K)\backslash
G(\mathbf{A}),\tau'_\infty\right),
\end{equation}
the space of functions $f$ on $G(\mathbf{A})$ taking values in the
vector space underlying $\tau'_\infty$ which are
left-$G(K)$-invariant and which satisfy
$f(xg_{\R}^{-1})=\tau_\infty'(g_\R)f(x)$ for $x\in G(\mathbf{A})$,
$g_{\R}\in G(K\otimes \R)$. This space has a left action of
$G(\mathbf{A}_f)$ via $(gf)(x)=f(xg)$; the automorphic
representations of $G$ which are $\tau'_\infty$ at the infinite
places are exactly the irreducible $G(\mathbf{A}_f)$-stable
subrepresentations of $M(\tau_\infty')$. Because an element of
$M(\tau_\infty')$ is determined by its restriction to
$G(\mathbf{A}_f)$, we may rewrite Eq (\ref{modforms1}) as
\begin{equation}
\label{modforms2} M(\tau'_\infty)=\mathcal{C}_{G(K)}\left(
G(\mathbf{A}_f),\tau'_\infty\right);
\end{equation}
that is, the space of functions $f$ on $G(\mathbf{A}_f)$ with values
in $V(\tau'_\infty)$ satisfying $f(kg)=\tau'_\infty(k)f(g)$ for
$k\in G(K)$ and $g\in G(\mathbf{A}_f)$.

Let $T$ be a set of double coset representatives for the finite set
$G(K)\backslash G(\mathbf{A}_f)/\hat{\OO}^*$.  Then $T$ is finite by ~\cite{Vigneras}, Theoreme 5.4.  For each $t\in T$,
there is a corresponding maximal order $\OO_t=t\hat{\OO}t^{-1}\cap
D$.  Then the right-hand side of Eq (\ref{modforms2}) decomposes
into a sum of $\hat{\OO}^*$-stable spaces indexed by $T$:
\begin{eqnarray*}
M(\tau'_\infty) &=&
\bigoplus_{t\in T}\mathcal{C}_{G(K)}\left(G(K)t\hat{\OO}^*,\tau'_\infty\right);
\end{eqnarray*}
with some thought each summand on the right is seen to be isomorphic
as a left $\hat{\OO}^*$-module to
$\Ind_{\OO_t^*}^{\hat{\OO}^*}(\tau_\infty')^\vee$, where $\OO_t^*$
is to be regarded as a subgroup of $\hat{\OO}^*$ via conjugation by
$t$, and $(\tau_\infty')^\vee$ is the contragradient of
$\tau_\infty'$.


By Frobenius reciprocity, the desired multiplicity $\mu(\tau)$ is
therefore a sum of terms
$$\mu(\tau)=\sum_{t\in T}\mu_t,$$ where $\mu_t$ is the inner product
of the $\OO_t^*$-modules $\tau_f'\vert_{\OO_t^*}$ and
$(\tau'_\infty)^\vee\vert_{\OO_t^*}$.  That is, $\mu_t$ is the
multiplicity of the trivial character in the restriction of
$\tau'=\tau'_f\otimes\tau'_\infty$ to $\OO_t^*$. By
Eq.~\ref{condition2}, $\tau$ is trivial on $\OO_K^*$;  it therefore
factors through the finite group $W_t=\OO_t^*/\OO_K^*$. Let $e_t$ be
the order of $W_t$. Then

\begin{equation}
\label{mut} \mu_t=\frac{1}{e_t}\sum_{\omega\in W_t}\tr \tau'(\omega)
\end{equation}

We claim that the term with $\omega=1$ dominates the sum in Eq.
(\ref{mut}). Indeed, suppose $\omega\in\OO_t^*$ is {\em outside} of
$\OO_K^*$.  Since $\omega^{e_t}$ belongs to the center of $G(K)$ but $\omega$
itself does not, $\omega$ is semisimple.  Therefore by
Lemma~\ref{tracebound2},
\begin{equation}
\label{firstineq} \abs{\tau_f(\omega)}\leq
C_12^{\nu(\tau)-n_{\text{sp}(\tau)}}
\end{equation}
for a constant $C_1$. {\em A priori} this $C_1$ depends on $\omega$, but
since there are only finitely many $\omega$ under consideration, we may
take $C_1$ to depend only on $K$.

We now turn to the infinite places.  For each infinite place $v$, let $\iota_v\from K_v\isom\R$ be the corresponding isomorphism.  We have that $G(K_v)$ is isomorphic to the group of quaternions $\alpha+\beta j$, with $\alpha,\beta$ complex numbers which are not both zero.  Let $\rho_v\from G(K_v)\to \GL(2,\C)$ be the representation
\begin{equation}
\label{rhov}
\alpha+\beta j\mapsto \tbt{\alpha}{\beta}{-\overline{\beta}}{\overline{\alpha}},
\end{equation} so that the central character of $\rho_v$ is $\iota_v$.  Suppose that $\tau_v$ is the discrete series representation $\mathcal{D}_{k_v,w}$ of $\GL(2,\R)$. (Note that the integer $w$ will not depend on the place $v$.) Then
$$\tau_v'=\left(\iota_v\circ\N_{D_v/K_v}\right)^{1-\frac{k+w}{2}}\otimes\Sym^{k_v-2}\rho_v,$$ where $\N_{D_v/K_v}$ is the reduced norm from $D_v$ to $K_v$;  see~\cite{Carayol:ladicreps2}, 0.10.

 Now suppose that $\omega$ is an element of $\OO_t^*\backslash\OO_K^*$.  Let $\omega_v$ be the image of $\omega$ under $\rho_v$. Since $\omega_v^{e_t}$ is a
scalar but $\omega_v$ is not, the eigenvalues of $\omega_v$ are $\zeta\alpha$
and $\zeta^{-1}\alpha$ for $\zeta\neq\pm 1$ satisfying
$\zeta^{e_t}=1$ and $\alpha^2=\det w_v=\iota\left(\N_{D_v/K_v}\omega\right)$. The trace of
$\Sym^{k-2}\tau_v'(\omega)$ is then
$(\zeta^k+\zeta^{k-2}+\dots+\zeta^{-k})\iota_v\left(\N_{D_v/K_v} \omega\right){(k-2)/2}$. The sum
of roots of unity is bounded by a constant $C_2$ which depends only
on $e_t$, and since $t$ runs over a finite set, this constant may be
taken to depend only on $K$. Therefore for each $v\vert\infty$ we
have
\begin{equation}
\label{realineq} \abs{\tr\tau'_v(\omega)}\leq C_2\iota_v\left(\N_{D_v/K_v}\omega\right)^{-w/2}
\end{equation}
Multiplying Eq.~\ref{realineq} for $v\in S_\infty$ gives
\begin{equation}
\label{secondineq} \abs{\tr\tau'_\infty(\omega)} \leq
C_2^n\left(\prod_{v\vert\infty}\N_{D_v/K_v}\omega\right)^{-w/2}=C_2^n,
\end{equation}
 where in the last step we have used the product formula together with the fact
that the reduced norm of $\omega$ belongs to $\OO_K^*$.  Putting together
Eqs.~\ref{firstineq} and~\ref{secondineq} we find the bound
\begin{equation}
\label{firstbound} \abs{\tr\tau'(\omega)}=\abs{\tr
\tau(\omega)\tr\tau_\infty'(\omega)}\leq
C_1C_2^n2^{\nu(\tau)-n_{\text{sp}}(\tau)}. \end{equation} Applying
Eq. (\ref{firstbound}) to the non-identity elements of the sum in
Eq. (\ref{mut}) gives
\begin{equation}
\label{mutbound}
\mu_t=\frac{1}{e_t}\dim\tau'+O\left(2^{\nu(\tau)-n_{\text{sp}}(\tau)}\right);
\end{equation}
summing this over $t\in T$ gives
\begin{equation}
\label{mubound} \mu(\tau)=\sum_{t\in T} \mu_t = \sum_{t\in
T}\frac{1}{e_t}\dim\tau' +
O\left(2^{\nu(\tau)-n_{\text{sp}}(\tau)}\right)
\end{equation}
Here we apply the ``mass formula" (see~\cite{Vigneras}, p. 142,
Corollaire 2.3):
$$\sum_{t\in T}\frac{1}{e_t} = 2^{1-n}\abs{\zeta_K(-1)}h.$$
Applying this to Eq. (\ref{mubound}) gives
\begin{equation}
\label{mubound2}
\mu(\tau)=2^{1-n}\abs{\zeta_K(-1)}h\dim\tau'+O\left(2^{\nu(\tau)-n_{\text{sp}}(\tau)}\right).
\end{equation}

As mentioned earlier, it is not necessarily true that
$\mu(\tau)=\#S(\tau)$. Indeed, if $v$ is a place at which $\tau_v$
is special, then $\mu(\tau)$ counts automorphic representations
which are principal series as well as special, owing to the
exceptional clause of Theorem~\ref{typetheorem}:  Types inside of
special representations are also contained in principal series
representations.  We adjust for this possibility using the
inclusion-exclusion principle. Let $S_{\text{sp}}$ be the
set of places at which $\tau$ is special. For $v\in S_{\text{sp}}$,
suppose $\tau_v^{\text{ps}}\in\Types(K_v)$ is the type of any
principal series representation whose restriction to
$\GL(2,\OO_{K,v})$ contains $\tau_v$.  Then $\tau_v^{\text{ps}}$ is
nothing but the central character of $\tau_v$.

For each subset $Y$ of $S_{\text{sp}}$, let $\tau^Y\in\Types(K)$
have the same components as $\tau$, but with $\tau_v^{\text{ps}}$ in
place of $\tau_v$ for each $v\in Y$.  Note that for $v\in
S_{\text{sp}}$, we have $\dim\tau_v=q_v$, so that
\begin{equation}
\label{tauy} \dim(\tau^Y)'=\dim\tau'\prod_{v\in Y}q_v^{-1}.
\end{equation}

Letting $\kappa=2^{1-n}\zeta_K(-1)h$, we have the following
expression for $S(\tau)$:

\begin{align*}
S(\tau) & = \sum_{Y\subseteq S_{\text{sp}}}(-1)^{\# Y}\mu(\tau^Y) &  \\
 &= \sum_{Y\subseteq S_{\text{sp}}}(-1)^{\#
Y}\left(\kappa\dim(\tau^Y)'+O\left(2^{\nu(\tau)-\#S_{\text{sp}}}\right)\right)
& \text{ by Eq. (\ref{mubound2}) applied to $\tau^Y$}\\
&= \sum_{Y\subseteq
S_{\text{sp}}}(-1)^{\#Y}\kappa\dim\tau'\prod_{v\in
Y}q_v^{-1}+O(2^{\nu(\tau)}) &\text{by Eq. (\ref{tauy})}\\
&=\kappa\dim\tau'\prod_{v\in S_{\text{sp}}}\left(1-\frac{1}{q_v}\right)+O(2^{\nu(\tau)})&\\
&=\kappa d(\tau)+O\left(2^{\nu(\tau)}\right)&\text{by
Eq.~(\ref{dvsdim})},
\end{align*}
thus completing the proof of Theorem~\ref{maintheorem} when $[K:\Q]$
is even.

\section{Proof in the case of $[K:\Q]$ odd}

\subsection{Shimura Curves}  In this section $n=[K:\Q]$ is odd.  Label the infinite places of $K$
as $v_0,\dots,v_{n-1}$.  Let $D/K$ be a quaternion algebra ramified
exactly at $S_\infty\backslash\set{v_0}$, let $G$ be the inner form
of $\GL(2)_K$ corresponding to $D$, and let $\OO$ be a maximal order
of $D$. Keep the notations $\mathbf{A}$ and $\mathbf{A}_f$ from the
preceding section.  For simplicity we will assume at first that
$K\neq\Q$, and then explain how to modify the proof in the case of
$K=\Q$ at the end of the section.

Let $X=\C\backslash\R$.  For an ideal $N$ of $\OO_K$, the Shimura curve $X_D(N)$ of (full)
level $N$ is the one whose complex points are
$$X_D(N)(\C)=G(K)\backslash\left(X\times G(\mathbf{A}_f)\right)/U_N,$$
where $U_N\subset \hat{\OO}^*$ is the compact
open subgroup consisting of elements congruent to the identity
modulo $N$.  Then $X_D(N)$ admits an action of $\hat{\OO}^*/U_N\isom \GL(2,\OO_K/N)$.

Now suppose $\tau\in\Types(K)$ is a global type. For $0\leq i\leq
n-1$, suppose that $\tau_{v_i}\isom \mathcal{D}_{k_v,w}$, and let $\mathbf{k}=(k_0,\dots,k_{n-1})$. As in the
case of $n$ even, we define a representation $\tau'_f$ of
$\hat{\OO}^*$ as well as a representation $\tau'_\infty$ of
$G(K\otimes\R)$, the only difference being that $\tau'_{v_0}$ is
essentially the same as $\tau_{v_0}$ because $G$ is split at that
place.  Let also $N$ be the level of $\tau'_f$.  As in the previous section, let $\mu(\tau)$ be the number of cuspidal automorphic representations $\pi$ of $G$ containing $\tau'=\tau'_f\otimes\tau'_\infty$.  The strategy is to determine $\mu(\tau)$ by computing the multiplicity of $\tau'$ in the appropriate cohomology group of the Shimura curve $X_D(N)$.

We proceed as in~\cite{Carayol:ladicreps}, 2.1.2, by defining a complex vector bundle $\mathcal{L}/X_D(N)$ analytically by
$$\mathcal{L}=G(K)\backslash(X\times G(\mathbf{A}_f)\times
V)/U_N.$$
Here $V=\bigotimes_{i=0}^{n-1}V_i$ is a certain
representation of $G(K\otimes\R)=\prod_{i=0}^{n-1}G(K_{v_i})$;  for
$i>0$, $V_i$ is $\tau'_v$, while for $i=0$, $V_i$ is a twist of the
$(k-2)$nd symmetric power of the tautological representation of
$G(K_{v_0})=\GL(2,\R)$ on $\C^2$.  Note that the vector bundle $\mathcal{L}$ is equivariant under the action of $\GL(2,\OO/N)$, so that the cohomology $H^1(X_D(N),\mathcal{L})$ admits an action of this group.  We now appeal to~\cite{BorelWallach}, VII, 3.2 (or for our particular application, see~\cite{Carayol:ladicreps}, 2.2.4) to compute the de Rham cohomology of $\mathcal{L}$:
\begin{equation}
\label{LieCoho}
H^1(X_D(N),\mathcal{L})\isom \bigoplus_{\pi} H^1(\mathfrak{g},K_\infty,\pi_\infty\otimes V)\otimes (\pi_f)^{U_N}.
\end{equation}
Here $\mathfrak{g}$ is the Lie algebra of $G(K\otimes\R)$ and $K_\infty\subset G(K\otimes\R)$ is the stabilizer of a point of $X$.  The sum runs over cuspidal automorphic representations $\pi$ of $G$ for which the summand is nonzero.  It follows from~\cite{RogawskiTunnell}, Prop. 1.5,  that $H^1(\mathfrak{g},K_\infty,\pi_\infty\otimes V)$ is zero unless $\pi_\infty\isom\tau'_\infty$, in which case it has dimension 2.  The isomorphism in Eq.~\ref{LieCoho} being $\GL(2,\OO_K/N)$-equivariant, we find that
\begin{equation}
\label{half}
\mu(\tau)=\frac{1}{2}\class{\tau'_f,H^1(X_D(N),\mathcal{L})}_{\GL(2,\OO_K/N)}.
\end{equation}

We now decompose the curve $X_D(N)$ into pieces which are stable
under the action of $(\OO/N\OO)^*$.  Let $T=G(K)\backslash
G(\mathbf{A}_f)/\hat{\OO}^*$.  There is a bijection between $T$ and
the class group $K^*\backslash\mathbf{A}_f^*/\hat{\OO}_K^*$ of $K$,
see~\cite{Vigneras}, Corollaire 5.7, part (i).  We find that
\begin{eqnarray*}
X_D(N)(\C) &=& \coprod_{t\in T} G(K)\backslash (X\times
G(K)t\hat{\OO}^*)/U_N\\
&=&\coprod_{t\in T} \OO_t^*\backslash (X\times\hat{\OO}^*)/U_N,
\end{eqnarray*}
where $\OO_t=t\hat{\OO}t^{-1}\cap D$ is to be considered a subgroup
of $\hat{\OO}^*$ via conjugation by $t$.  Write $X_{D,t}(N)$ for
$\OO_t^*\backslash (X\times\hat{\OO}^*)/U_N$; it is a Riemann
surface with an action of
$\hat{\OO}^*/U_N\isom\GL(2,\OO_K/N)$. Suppose $\Gamma_t(1)$
is the image in $\GL(2,\R)$ of the group of units of $\OO_t^*$ and
that $\Gamma_t(N)\subset\Gamma_t(1)$ arises from the subgroup of
elements congruent to 1 modulo $N$.  It is easy to check that
\begin{equation}
\label{YDt} X_{D,t}(N)=\hat{\OO}^*/U_N
\times_{\Gamma_t(1)/\Gamma_t(N)} \Gamma_t(N)\backslash X.
\end{equation}
>From Eq.~\ref{YDt} we have the isomorphism of
$\hat{\OO}^*/U_N$-modules
$$H^1(X_{D,t},\mathcal{L})=\Ind_{\Gamma_t(1)/\Gamma_t(N)}
H^1(\Gamma_t(N)\backslash X,\mathcal{L});$$ by abuse of notation we
have written $\mathcal{L}$ for the restriction of that vector bundle
to both $X_{D,t}(N)$ and $\Gamma_t(N)\backslash X$.  Let $\mu_t$ be
the multiplicity of $\tau'_f$ inside $H^1(X_{D,t},\mathcal{L})$ as an
$\hat{\OO}^*/U_N$-module, so that by Eq.~\ref{half} $2\mu(\tau)=\sum_{t\in T}\mu_t$.
By Frobenius reciprocity
\begin{equation}
\label{mutauf}
\mu_t=\class{\tau'_f\vert_{\Gamma_t(1)/\Gamma_t(N)},H^1(\Gamma_t(N)\backslash
X,\mathcal{L})}_{\Gamma_t(1)/\Gamma_t(N)}.
\end{equation}

\subsection{Cohomology of discrete groups}

It therefore remains to compute the structure of
$H^1(\Gamma_t(N)\backslash X,\mathcal{L})$ as a
$\Gamma_t(1)/\Gamma_t(N)$-module.  Assume that $N$ is large enough
so that every nonscalar element of $\Gamma_t(N)$ acts without fixed
points on $X$.  As $\Gamma_t(N)\backslash X$ is compact,
$H^1(\Gamma_t(N)\backslash X,\mathcal{L})\isom H^1(\Gamma_t(N),V)$.
It will suffice to compute the Euler characteristic
$$\chi(\Gamma_t(N),V)=\sum_{i=0}^2(-1)^i[H^i(\Gamma_t(N),V)]$$ in the Grothendieck group of
$\Gamma_t(1)/\Gamma_t(N)$.  For this we have the:

\begin{lemma}
\label{fuchsian} Let $\Gamma\subset\GL(2,\R)$ be a discrete
subgroup, acting on $X$ with compact quotient. Let $x_1,\dots,x_r\in X$ be a complete set of $\Gamma$-inequivalent fixed points having fixed subgroups
$\Gamma_1,\dots,\Gamma_r$. Suppose $\Gamma'\subset\Gamma$ is a
normal subgroup acting without fixed points on $X$.  Let
$G=\Gamma/\Gamma'$ and let $Z\subset\Gamma$ be the group of elements
of $\Gamma$ which are scalar. Suppose $V$ is a finite-dimensional
complex vector space admitting an action of $\Gamma$ such that
scalar matrices in $\Gamma'$ act trivially on $V$. Then we have the
equality:
\begin{equation}
\label{fuchsianeq}
\chi(\Gamma',V)=\left(\chi(\Gamma\backslash X)-r\right)\Ind_Z^GV\vert_Z+\sum_{i=1}^r \Ind_{\Gamma_i}^G
V\vert_{\Gamma_i}
\end{equation} in the Grothendieck group of $G$.
\end{lemma}

\begin{proof}
We construct a simplicial complex $\mathcal{K}$ with underlying topological
space $X$ in such a way that $\Gamma$ acts on $\mathcal{K}$ and such that
each elliptic fixed point $x_j$ is a vertex of $\mathcal{K}$.  For $i=0,1,2$
let $C_i(\mathcal{K})$ be the free $C$-vector space with basis the $i$-cells
of $K$, and let $N_i=\dim C_i(\mathcal{K})$. Let $C^i(\mathcal{K},V)$ be the space of
$\C[\Gamma']$-module homomorphisms $C_i(\mathcal{K})\to V$.  Then $C^i(\mathcal{K},V)$
carries a left $\Gamma$-module structure: if $\gamma\in\Gamma$ and
$f\in C^i(\mathcal{K},V)$, then $(\gamma f)(x)=\gamma(f(\gamma^{-1}(y)))$. The
action factors through an action of $G$.  It is standard
(see~\cite{Shimura:Automorphic}, Prop. 8.1) that the cohomology of
the complex $$0\to C^0(\mathcal{K},V)\to C^1(\mathcal{K},V)\to C^2(\mathcal{K},V)\to 0$$ agrees
with $H^i(\Gamma',V)$.  Therefore $\chi(\Gamma',V)$ is the
alternating sum of the $C^i(\mathcal{K},V)$ in the Grothendieck ring of $G$.

We wish to compute the structure of $C^0(\mathcal{K},V)$ as an $\C[G]$-module.
Let $x_1,\dots,x_{N_0}$ be a complete set of $\Gamma$-inequivalent
$0$-cells of $\mathcal{K}$, with the fixed points $x_1,\dots,x_r$ as in the
hypothesis.   Writing $W_i$ for the $\C$-span of the $\Gamma$-orbit
of $x_i$, we have a decomposition of $\C[\Gamma]$-modules
$C_0(\mathcal{K})=\bigoplus_{i=1}^{N_0}W_i$ of $C_0(\mathcal{K})$, and therefore
$$C^0(\mathcal{K},V)=\bigoplus_{i=1}^{N_0}\Hom_{\Gamma'}(W_i,V).$$  Note that
$W_i\approx\Ind_{\Gamma_i}^\Gamma 1$ for $i\leq r$, and that
$\Gamma_i=Z$ for $i\geq r+1$. For each $i$, the $\C[G]$-module
$\Hom_{\Gamma'}(W_i,V)$ can be modeled on the space of functions
$\Phi\from\Gamma\to V$ satisfying $\Phi(\gamma'\gamma
k)=\rho(\gamma')\Phi(\gamma)$ whenever $\gamma'\in\Gamma'$,
$\gamma\in\Gamma$, $k\in\Gamma_i$. The action of $G$ is given by
$\overline{g}(\Phi)(\gamma)=\rho(g)\Phi(g^{-1}\gamma)$ whenever
$\overline{g}=\Gamma'g\in G$.

We claim that there is an isomorphism of $G$-modules
$$\Hom_{\Gamma'}(W_i,V)\tilde{\longrightarrow}\Ind_{\Gamma_i}^G V\vert_{\Gamma_i}.$$
The space on the right is modeled on the space of functions
$\Psi\from G\to V$ satisfying
$\Psi(kg\Gamma')=\rho(k)\Psi(g\Gamma')$ for all $k\in \Gamma_i$,
$g\Gamma\in G$.  Choose a set of representatives $g_1,\dots,g_s$ for
$G/\Gamma_i=\Gamma'\backslash\Gamma/\Gamma_i$;  it is easily checked
that an isomorphism is given by $\Phi\mapsto\Psi$, where
$\Psi(\Gamma'g_j)=g_j\Phi(g_j^{-1})$.

Thus in the Grothendieck group of $G$ we have

\begin{eqnarray*}
C^0(\mathcal{K},V)&=&\sum_{i=1}^{N_0}\Ind_{\Gamma_i}^G V\vert_{\Gamma_i} \\
&=& (N_0-r)\Ind_Z^G V\vert_Z +\sum_{i=1}^r\Ind_{\Gamma_i}^G
V\vert_{\Gamma_i}.
\end{eqnarray*}

A similar calculation holds for $C^1(\mathcal{K},V)$ and $C^2(\mathcal{K},V)$, except
that there are no cells fixed by elements of $\Gamma$.   Therefore:

\begin{eqnarray*}
C^1(\mathcal{K},V)&=& N_1\Ind_Z^G V\\
C^2(\mathcal{K},V)&=& N_2\Ind_Z^G V
\end{eqnarray*}

Taking the alternating sum of the $C^i(\mathcal{K},V)$ gives the expression in
the Lemma.
\end{proof}

We apply Lemma~\ref{fuchsian} to the groups
$\Gamma_t(N)\subset\Gamma_t(1)$ and to the representation $V$ of
$\Gamma_t(1)$.  Keep the notations $\Gamma_i$, $x_i$, and $r$ from
the lemma.  Since the multiplicity of $\tau$ in $H^0$ and $H^2$ is
bounded, the multiplicity of $\tau$ in $\bigoplus_{t\in
T}\chi(\Gamma_t(N),V)$ is $-\sum_{t\in T}\mu_t=-2\mu(\tau)$ to
within an error term depending only on $K$. From Eq. (\ref{mutauf})
we have
\begin{equation}
-\mu_t=\left(\chi(X_t(1))-r\right)\dim
V\dim\tau'_f+\sum_{i=1}^r\class{\tau'_f\vert_{\Gamma_i},V\vert_{\Gamma_i}}+O(1)
\end{equation}
Let $W_i=\Gamma_i/Z$ have order $e_i$.  Since $\tau_f'$ and $V$ have
the same values on $Z=\OO_K^*$, the summand on the right is
$$\class{\tau'_f\vert_{\Gamma_i},V\vert_{\Gamma_i}}=\frac{1}{e_i}\sum_{w\in
W_i}\tr\tau'_f(w)\overline{\tr(w\vert V)}.$$  By an argument of the
same sort as in the previous section, this quantity is
$\frac{1}{e_i}\dim \tau'_f\dim
V+O\left(2^{\nu(\tau)-n_{\text{sp}}(\tau)}\right)$. Therefore

$$-2\mu_t=\left[\chi(X_h(1))-r+\sum_{i=1}^r
\frac{1}{e_i}\right]\dim\tau'_f\dim V
+O\left(2^{\nu(\tau)-n_{\text{sp}}(\tau)}\right).$$ The expression
in square brackets is
\begin{eqnarray*}
\chi(X_t(1))-\sum_{i=1}^r\left(1-\frac{1}{e_i}\right) &=&
-\frac{1}{2\pi}\vol(\Gamma_t(1)\backslash X)\\
&=&-\abs{\zeta_K(-1)}2^{2-n}
\end{eqnarray*}
from~\cite{Vigneras}, p. 109, Exemple 5, together with p. 119, Prop.
2.10.  Since $2\mu(\tau)=\sum_{t\in T}\mu_t$ and $\#T = h$, the
argument continues exactly as in the previous section.

\subsection{The case of $K=\Q$.} \label{Q} Now suppose $K=\Q$. Let
$\tau=\tau_f\otimes\tau_\infty\in\Types(\Q)$.  Let $N$ be the level
of $\tau$, considered as a rational integer, and suppose the infinite component of $\tau$ is $\tau_\infty=D_{k,w}$.  .  Assume $N\geq 22$.  We note that condition (2)
in the definition of global types reduces to the condition that the
central character of $\tau_f$ take the value $(-1)^k$ at $-1$.

Let $\mu(\tau)$ be the number of cuspidal automorphic representations $\pi$ of $\GL(2,\mathbf{A}_\Q)$ containing $\tau=\tau_f\otimes\tau_\infty$.  The analogue of Eq.~\ref{LieCoho} in the case of $K=\Q$ is found in~\cite{LanglandsAntwerp}, Thm. 2.10.  The result is the same except that parabolic cohomology must be used.  Let $Y_D(N)$ be the (non-complete) Shimura curve for the split algebra $D=M_2(\Q)$.  Let $\mathcal{L}/Y_D(N)$ be the vector bundle corresponding to $\tau_\infty$;  then
\begin{equation}
\label{halfQ}
\mu(\tau)=\frac{1}{2}\class{\tau_f,H^1_P(X_D(N),\mathcal{L})}_{\GL(2,\Z/N\Z)}.
\end{equation}  The curve $Y_D(N)$ has connected components, each of which is the classical modular curve $Y(N)$:    $$Y_D(N)=\GL(2,\Z/N\Z)\times_{\SL(2,\Z/N\Z)} Y(N).$$  Therefore $$\mu(\tau)=\frac{1}{2}\class{\tau_f\vert_{\SL(2,\Z/N\Z)},H^1_P(Y(N),\mathcal{L})}.$$

Our goal is therefore to determine the class of $H^1_P(Y(N),\mathcal{L})$ in the Grothendieck group of $\SL(2,\Z/N\Z)$.  Since $N>1$, the group $\Gamma(N)$ acts on the upper half-pane $\mathcal{H}$ without fixed points and we have the $S_n$-equivariant isomorphism
\begin{equation}
H^1_P(Y(N),\mathcal{L})\isom H^1_P(\Gamma(N),V_k),
\end{equation}
where $V_k=\Sym^{k-2}V_3$ is the $(k-2)$nd symmetric power of the tautological representation  of $\SL(2,\Z)$ on $\C^2$.  (For the definition of parabolic cohomology of a Fuchsian group, see~\cite{Shimura:Automorphic}, 8.1.)

We now compute the $S_N$-module $H^1_P(\Gamma(N),V_k)$.  The calculation hinges on the geometry of the Galois cover
$X(N)\to X(1)$ of (complete) modular curves.  Let $S_N=\SL(2,\Z/N\Z)$; then
the Galois group of this cover is $S_N/\set{\pm I}$.  The cover is
branched over three points in $X(1)$, namely the images in $X(1)$ of
the points $\rho=e^{2\pi i/3}$, $i$, and $\infty$ of the upper half
plane $\mathcal{H}$.  Those points have the following stabilizers in $\SL(2,\Z)$:
\begin{eqnarray*}
\Gamma_\rho&=&\class{\tbt{0}{-1}{1}{1}}\\
\Gamma_i&=&\class{\tbt{0}{-1}{1}{0}}\\
\Gamma_\infty &=&\class{\pm\tbt{1}{1}{}{1}}.
\end{eqnarray*}
For $j\in\set{\rho,i,\infty}$, let $\overline{\Gamma}_j$ be the
image of $\Gamma_j$ in $S_N$.   The center of $\SL(2,\Z)$ is
$Z=\set{\pm 1}$; let $\overline{Z}$ be its image in $S_N$.   Since
$N>2$, the reduction map $Z\to\overline{Z}$ is an isomomorphism, and
therefore any $Z$-module can be considered an $\overline{Z}$-module.
The same is true for the groups $\Gamma_j$ and
$\overline{\Gamma}_j$.

To compute the structure of $H^1_P(\Gamma(N),V_k)$ as an $S_N$-module, we must modify Lemma~\ref{fuchsian} to include a term coming from the
unique cusp of $X(1)$.  Let $\chi_P(\Gamma(N),V_k)$ be the alternating sum of the parabolic cohomology groups $H^i_P(\Gamma(N),V_k)$ in the Grothendieck group of $S_N$.  Let $\sgn$ denote the obvious character of $\overline{Z}$ and of $\overline{\Gamma}_\infty$.

\begin{lemma}
\label{fuchsianQ}
$$\chi_P(\Gamma(N),V_k)=-\left[\Ind_{\overline{Z}}^{S_N} V_k\vert_Z\right] +
\left[\Ind_{\overline{\Gamma}_\rho}^{S_N}
V_k\vert_{\Gamma_\rho}\right]
+\left[\Ind_{\overline{\Gamma}_i}^{S_N} V_k\vert_{\Gamma_i}\right]
+ \left[\Ind_{\overline{\Gamma}_\infty}^{S_N} \sgn^k\right]$$
\end{lemma}

\begin{proof} Remove a small open disc containing $\infty$ from the projective line $X(1)$, and let $\mathcal{H}^o$ be the preimage of the result.  This can be accomplished by removing from $\mathcal{H}$ all the $\SL(2,\Z)$-translates of the region $\set{x+iy\vert y\geq y_0}$ for $y_0$ large enough. Construct an $\SL(2,\Z)$-stable simplicial complex $\mathcal{K}$ whose underlying topological space is $\mathcal{H}^o$ such that $i$ and $\rho$ are vertices of $\mathcal{K}$.  Let $u=\tbt{1}{1}{0}{1}$, so that $\Gamma_\infty\cap\Gamma(N)$ is generated by $u^N$.  Assume there is a 1-cell $t_1$ in $\mathcal{K}$ for which $\partial t_1$ is of the form $(u-1)z$ for a vertex $z$ of $\mathcal{K}$.  Then the 1-chain $\gamma:=\sum_{j=0}^{N-1} u^jt_1$ has boundary $(u^N-1)z$.  The boundary of the quotient complex $\Gamma(N)\backslash\mathcal{K}$ is exactly the set of translates $g\gamma$, where $g$ runs over a set of coset representatives for $\Gamma(N)\backslash\SL(2,\Z)$.  Note that $\partial(g\gamma)=(gu^Ng^{-1}-1)gz$.    Let $C_P^1(\mathcal{K},V_k)\subset C^1(\mathcal{K},V_k)$ be the space of $\C[\Gamma(N)]$-module homomorphisms $C_1(\mathcal{K})\to V_k$ mapping the 1-cycle $g\gamma$ into $(gu^Ng^{-1}-1)V_k$ for each $g\in\SL(2,\Z)$.  Then the complex $$1\to C^0(\mathcal{K},V_k)\to C^1_P(\mathcal{K},V_k)\to C^2(\mathcal{K},V_k)\to 1$$ has cohomology $H^*(\Gamma(N),V_k)$ ({\em loc. cit.}, Prop. 8.1).  We remark that $(u^N-1)V_k=(u-1)V_k$ has codimension 1 in $V_k$.

Let $t_1,t_2,\dots,t_{N_1}$ be a complete set of $\SL(2,\Z)$-inequivalent 1-cells of $\mathcal{K}$.  Writing $W_j$ for the $\C$-span of the $\SL(2,\Z)$-orbit of $t_j$, we have an isomorphism of $\SL(2,\Z)$-modules $W_j\isom \Ind_{Z}^{\SL(2,\Z)} 1$.  Let $C^1_P(W_j,V_k)\subset C^1_P(\mathcal{K},V_k)$ be the $\SL(2,\Z)$-submodule of cocycles supported on $W_j$, so that $C_1^P(\mathcal{K},V_k)=\bigoplus_{j=1}^{N_1}C^1_P(W_j,V_k)$.  Then for $j>1$ we have
$$
[C^P(W_j,V_k)]=\left[\Hom_{\Gamma(N)}(W_j,V_k)\right]=\left[\Ind_{\overline{Z}}^{S_N} V_k\vert_Z \right]
$$
as classes in the Grothendieck group of ${S_N}$.  On the other hand for $j=1$ we have the exact sequence of ${S_N}$-modules
\begin{equation}
\label{SESgamma}
0\to C^1_P(W_1,V_k) \to \Hom_{\Gamma(N)}(W_j,V_k) \stackrel{\xi}{\to} \Hom_{\overline{\Gamma}_\infty}(\C[{S_N}],V_k/(u-1)V_k) \to 0
\end{equation}
where $\xi(f)\from\C[{S_N}]\to V_k/(u-1)V_k$ is defined by $\xi(f)(\overline{g})=g^{-1}f(g\gamma)+(u-1)V_k$.  The homomorphism $\xi(f)$ is $\overline{\Gamma}_\infty$-equivariant because
\begin{eqnarray*}
\xi(f)(gu)&=&u^{-1}g^{-1}f(gu\gamma)\\
&=& u^{-1} g^{-1}f(gu\gamma) \\
&=& u^{-1} g^{-1}f(g(u^N-1)t_1 + g\gamma)\\
&=& u^{-1} g^{-1} (u^N-1) f(gt_1) + u^{-1}g^{-1} f(g\gamma)\text{ because $u^N-1\in \C[\Gamma(N)]$}\\
&\equiv& g^{-1}f(g\gamma) \equiv \xi(f)(g)\pmod{(u-1)V_k}.
\end{eqnarray*}
Since $[V_k/(u-1)V_k]=[\sgn^k]$ in the Grothendieck group of $\overline{\Gamma}_\infty$, Eq.~\ref{SESgamma} implies that $$[C^1_P(W_1,V_k)] = \left[\Ind_{\overline{Z}}^{S_N} V_k\vert_Z\right] -          \left[\Ind_{\overline{\Gamma}_\infty}^{S_N} \sgn^k\right]$$ and therefore that
$$[C^1_P(\mathcal{K},V_k)]=\bigoplus_{j=1}^{N_1}[C^1_P(W_j,V_k)]=N_1\left[\Ind_{\overline{Z}}^{S_N}V_k\vert_Z\right] -\left[\Ind_{\overline{\Gamma}_\infty}^{S_N}\sgn^k\right]$$ in the Grothendieck group of $S_N$.

Therefore the calculation of $\chi(\Gamma(N),V_k)$ proceeds as in Lemma~\ref{fuchsian} with the only change being that there is a contribution of $\left[\Ind_{\overline{\Gamma}_\infty}^{S_N}\sgn^k\right]$ coming from the space of 1-cochains.  Since $\chi(\SL(2,\Z)\backslash\mathcal{H}^o)=1$ and there are two $\SL(2,\Z)$-orbits of elliptic fixed points, the appropriate modification of Eq.~\ref{fuchsianeq} is
$$
\chi(\Gamma(N),V_k)= -\left[\Ind_{\overline{Z}}^{S_N} V_k\vert_{Z}\right]+\sum_{j\in\set{\rho,i}} \left[\Ind_{\overline{\Gamma_j}}^{S_N} V_k\vert_{\Gamma_j}\right] + \left[\Ind_{\overline{\Gamma}_\infty}^{S_N}\sgn^k\right]
$$ as required.\end{proof}

The relationship between $H^1_P(\Gamma(N),V_k)$ and $\chi_P(\Gamma(N),V_k)$ is given by
\begin{equation}
\label{Hversuschi}
\left[H^1_P(\Gamma(N),V_k)\right]=\begin{cases} -\chi_P(\Gamma(N),V_k)-2[1],& k=2,\\
-\chi_P(\Gamma(N),V_k),& k>2,\end{cases}
\end{equation}
the reason being that both $H^0_P(\Gamma(N),V_k)$ and $H^2_P(\Gamma(N),V_k)$ are one-dimensional if
$k=2$ and vanish if $k>2$.
We analyze each of the three terms appearing on the right-hand side
of Eq.~\ref{fuchsianQ}.   For the first term, note that
$V_k\vert_{Z}$ is simply $k-1$ copies of the sign character $\sgn^k$
of $Z$, so that
\begin{equation}
\label{term1} \left[\Ind_{\overline{Z}}^{S_N}
V_k\vert_{Z}\right]=(k-1)\left[\Ind_{\overline{Z}}^{S_N}
\sgn^k\right].
\end{equation}
For the second term of Eq.~\ref{fuchsianQ}, we have
$V_3\vert_{\Gamma_\rho}=\chi_\rho\oplus\chi^{-1}_\rho$, where
$\chi_\rho\from\Gamma_\rho\to \C^*$ is the character
$\tbt{}{-1}{1}{1}\mapsto e^{2\pi i/6}$.  Then
\begin{equation}
\label{Gamma0}
V_k\vert_{\Gamma_\rho}=\Sym^{k-2}V_3\vert_{\Gamma_\rho}=\chi_\rho^{k-2}\oplus\chi_\rho^{k-4}\oplus\dots\oplus\chi_\rho^{-k+2}
\end{equation}
Each character of $\Gamma_\rho$ of the same parity as $k$ appears in
the sum about $k/3$ times, so we expect the right side of
Eq.~\ref{Gamma0} to contain about $k/3$ copies of the sum
$\chi_\rho^k\oplus\chi_\rho^{k+2}\oplus\chi_\rho^{k+4}=\Ind_{Z}^{\Gamma_\rho}
\sgn^k$.  More precisely, we have the following relation in the
Grothendieck group of $\Gamma_\rho$:
\begin{equation}
\label{Gamma01}
V_k\vert_{\Gamma_\rho}=\floor{\frac{k}{3}}\left[\Ind_{Z}^{\Gamma_\rho}
\sgn^k\right]+\eps_\rho,
\end{equation}
where the error term is a virtual representation of $\Gamma_\rho$
given by
\begin{equation}
\label{error0}
\eps_\rho=\begin{cases} -[1],&k\equiv 0\pmod{6}\\
0,&k\equiv 1\pmod{6}\\
[1],&k\equiv 2\pmod{6}\\
-[\chi_\rho^3],&k\equiv 3\pmod{6}\\
0,&k\equiv 4\pmod{6}\\
[\chi_\rho^3],&k\equiv 5\pmod{6}.
\end{cases}
\end{equation}
For the third term of Eq.~\ref{fuchsianQ}, let
$\chi_i\from\Gamma_i\to\C^*$ be the character defined by
$\tbt{}{1}{-1}{}\to i$.  The analysis is similar to the case of
$\Gamma_\rho$. We have
\begin{equation}
\label{Gamma11}
V\vert_{\Gamma_i}=\floor{\frac{k}{2}}\left[\Ind_{Z}^{\Gamma_i}
\sgn^k\right]+\eps_i,
\end{equation}
where the error term is $$\eps_i=\begin{cases} -[1],& k\equiv 0 \pmod{4}\\
0,& k\equiv 1\pmod{4}\\ -[\chi_i^2],&k\equiv 2\pmod{4}\\ 0,& k\equiv
3\pmod{4}.\end{cases}$$

Let $f(k)=(k-1)-\floor{\frac{k}{3}}-\floor{\frac{k}{2}}$.
Substituting Eqs.~\ref{term1},~\ref{Gamma01}, and~\ref{Gamma11} into
Eq.~\ref{fuchsianQ} gives:
\begin{Theorem} For $k=2$ we have
\begin{equation}
\label{weight2}
[H^1_P(\Gamma(N),V_2)]-2[1]=\left[\Ind_{\overline{Z}}^{S_N} 1\right]
-\left[\Ind_{\overline{\Gamma}_\rho}^{S_N} 1 \right]
-\left[\Ind_{\overline{\Gamma}_i}^{S_N} 1\right]
-\left[\Ind_{\overline{\Gamma}_\infty}^{S_N} 1\right].
\end{equation}
For weight $k>2$,
\begin{equation}
\label{higherwt}
[H^1_P(\Gamma(N),V_k)]=f(k)\left[\Ind_{\overline{Z}}^{S_N} \sgn^k\right]
-\left[\Ind_{\overline{\Gamma}_\rho}^{S_N} \eps_\rho\right]
-\left[\Ind_{\overline{\Gamma}_i}^{S_N} \eps_i\right]
-\left[\Ind_{\overline{\Gamma}_\infty}^{S_N} \sgn^k\right].
\end{equation}
\end{Theorem}

We may now complete the proof of Theorem~\ref{maintheorem}.  Let $\tau\in\Types(\Q)$ be a type of level $N$ and weight $k$ and let $\mu(\tau)$ be its multiplicity in $H^1_P(\Gamma(N),V_k)$.  We claim that $\mu(\tau)=\frac{1}{6}(k-1)\dim\tau_f+O\left(2^{\nu(\tau)-n_{\text{sp}}(\tau)}\right)$.  (Note that $f(k)\sim (k-1)/6=2\zeta_\Q(-1)(k-1)$.)  The only feature separating this calculation from that of the previous section is the appearance of the term $\Ind_{\overline{\Gamma}_\infty}^{S_N} \sgn^k$.  By the Chinese remainder theorem the multiplicity of $\tau_f$ in this term is a product of local multiplicities $\class{\tau_p\vert_{U_p},1}$, where $U_p\subset\GL(2,\Z_p)$ is the unipotent subgroup.  A case-by-case analysis shows that the local multiplicity at $p$ is 1 if $\tau_p$ is special and at most 2 in any case.  (In fact it is 0 if $\tau_p$ is supercuspidal.)  Therefore the multiplicity of $\tau_f$ in $\Ind_{\overline{\Gamma}_\infty}^{S_N} \sgn^k$ is $O\left(2^{\nu(\tau)-n_{\text{sp}}(\tau)}\right)$.  The claim follows and the calculation proceeds exactly as before.

We remark without proof that $[H^1(\Gamma(N),V_k)]=2[S_k(\Gamma(N))]$, where $S_k(\Gamma(N))$ is the space of cusp forms of weight $k$ for $\Gamma(N)$.

We conclude the section with a table of types $\tau\in\Types(\Q)$ for which $S(\tau)=\emptyset$.  Table~\ref{typeslist} lists configurations of finite components $\tau_p$ together with those weights $k$ of the appropriate parity for which $\tau = \left(\bigotimes_p \tau_p\right)\otimes \mathcal{D}_{k,(-1)^k}$ is a global type without any matching cuspforms.   Unfortunately, the list omits local 2-adic types of conductor $2^n$, where $n\geq 5$ is odd; the presence of extraordinary (non-dihedral) Galois representations of $\Q_2$ in that case complicates matters considerably.  
The list is complete in the sense that any type $\tau\in\Types(\Q)$ with $S(\tau)=\emptyset$ is a twist of one of the listed types, unless that type should include one of the aforementioned 2-adic types.  The notations for local components are given in Table~\ref{legend}.

\thispagestyle{empty}

\begin{table}
\caption{Global inertial types over $\Q$ lacking representation by a cusp form.
For explanation of notation, see Table~\ref{legend}.}
\begin{center}
\begin{tabular}{| l | l | l || l | l | l || l | l | l |}
Cond. & Local Components & $k$ & Cond. & Local Components & $k$ \\ \hline
1 &  & 2,4,6,8,10,14
&
25 & $\SC_{5}(12)$ &  2
\\
2 & $\SP_{2}$ &  2, 4, 6
&
26 & $\SP_{2}$, $\PS_{13}(6)$ &  2
\\
3 & $\SP_{3}$ &  2, 4
&
27 & $\SC_{27}(\sqrt{3},1)$ &  2
\\
 & $\PS_{3}(2)$ &  3, 5
&
28 & $\SC_{2}(3)$, $\SP_{7}$ &  2
\\
4 & $\SC_{2}(3)$ &  2, 4, 8
&
36 & $\SC_{2}(3)$, $\SC_{3}(4)$ &  4
\\
 & $\PS_4$ &  3
&
 & $\SC_{2}(3)$, $\SC_{3}(8)$ &  3
\\
5 & $\SP_{5}$ &  2
&
45 & $\SC_{3}(4)$, $\SP_{5}$ &  2
\\
 & $\PS_{5}(2)$ &  2, 4
&
49 & $\SC_{7}(8)$ &  2
\\
 & $\PS_{5}(4)$ &  3
&
 & $\SC_{7}(24)$ &  2
\\
6 & $\SP_{2}$, $\SP_{3}$ &  2
&
50 & $\SP_{2}$, $\SC_{5}(8)$ &  5
\\
7 & $\SP_{7}$ &  2
&
 & $\SP_{2}$, $\SC_{5}(24)$ &  3
\\
 & $\PS_{7}(3)$ &  2
&
52 & $\SC_{2}(3)$, $\PS_{13}(2)$ &  2
\\
 & $\PS_{7}(6)$ &  3
&
54 & $\SP_{2}$, $\SC_{27}(\sqrt{-3},1)$ &  3
\\
8 & $\PS_8$ &  2
&
60 & $\SC_{2}(3)$, $\SP_{3}$, $\SP_{5}$ &  2
\\
 & $\SC_8$ &  2
&
64 & $\SC_{64}(3)$ &  2
\\
9 & $\SC_{3}(4)$ &  2, 6
&
 & $\SC_{64}(2)$ &  3
\\
 & $\SC_{3}(8)$ &  3
&
72 & $\SC_8$, $\SC_{3}(4)$ &  2
\\
 & $\PS_9$ &  2
&
90 & $\SP_{2}$, $\SC_{3}(4)$, $\PS_{5}(2)$ &  2
\\
10 & $\SP_{2}$, $\SP_{5}$ &  2
&
 & $\SP_{2}$, $\SC_{3}(8)$, $\SP_{5}$ &  3
\\
 & $\SP_{2}$, $\PS_{5}(2)$ &  2
&
98 & $\SP_{2}$, $\SC_{7}(4)$ &  2
\\
11 & $\PS_{11}(5)$ &  2
&
 & $\SP_{2}$, $\SC_{7}(12)$ &  2
\\
12 & $\SC_{2}(3)$, $\SP_{3}$ &  2, 6
&
100 & $\SC_{2}(3)$, $\SC_{5}(3)$ &  2
\\
13 & $\SP_{13}$ &  2
&
 & $\SC_{2}(3)$, $\SC_{5}(6)$ &  2
\\
 & $\PS_{13}(2)$ &  2
&
 & $\SC_{2}(3)$, $\SC_{5}(8)$ &  3
\\
 & $\PS_{13}(3)$ &  2
&
 & $\SC_{2}(3)$, $\SC_{5}(12)$ &  2
\\
14 & $\SP_{2}$, $\PS_{7}(3)$ &  2
&
108 & $\SC_{2}(3)$, $\SC_{27}(\sqrt{3},1)$ &  2
\\
15 & $\SP_{3}$, $\PS_{5}(2)$ &  2
&
121 & $\SC_{11}(12)$ &  2
\\
17 & $\PS_{17}(2)$ &  2
&
 & $\SC_{11}(60)$ &  2
\\
 & $\PS_{17}(4)$ &  2
&
126 & $\SP_{2}$, $\SC_{3}(4)$, $\SP_{7}$ &  2
\\
18 & $\SP_{2}$, $\SC_{3}(4)$ &  2, 4, 8
&
 & $\SP_{2}$, $\SC_{3}(8)$, $\PS_{7}(2)$ &  2
\\
 & $\SP_{2}$, $\SC_{3}(8)$ &  5
&
135 & $\SC_{27}(\sqrt{-3},1)$, $\SP_{5}$ &  2
\\
19 & $\PS_{19}(3)$ &  2
&
147 & $\SP_{3}$, $\SC_{7}(4)$ &  2
\\
20 & $\SC_{2}(3)$, $\PS_{5}(2)$ &  2
&
 & $\SP_{3}$, $\SC_{7}(12)$ &  2
\\
22 & $\SP_{2}$, $\SP_{11}$ &  2
&
150 & $\SP_{2}$, $\SP_{3}$, $\SC_{5}(3)$ &  2
\\
25 & $\SC_{5}(3)$ &  2
&
 & $\SP_{2}$, $\SP_{3}$, $\SC_{5}(6)$ &  2
\\
 & $\SC_{5}(6)$ &  2
&
 & $\SP_{2}$, $\SP_{3}$, $\SC_{5}(8)$ &  3
\\
 & $\SC_{5}(8)$ &  3
&
 & $\SP_{2}$, $\SP_{3}$, $\SC_{5}(12)$ &  2
\\

\end{tabular}
\end{center}
\label{typeslist}
\end{table}

\begin{table}

\begin{center}
\begin{tabular}{| l | l | l || l | l | l || l | l | l |}
Cond. & Local Components & $k$ & Cond. & Local Components & $k$ \\ \hline
162 & $\SP_{2}$, $\SC_{81}$ & 2
&
578 & $\SP_{2}$, $\SC_{17}(3)$ & 2
\\
180 & $\SC_{2}(3)$, $\SC_{3}(4)$, $\SP_{5}$ & 2
&
 & $\SP_{2}$, $\SC_{17}(6)$ & 2
\\
 & $\SC_{2}(3)$, $\SC_{3}(4)$, $\PS_{5}(2)$ & 2
&
 & $\SP_{2}$, $\SC_{17}(12)$ & 2
\\
192 & $\SC_{64}(1)$, $\SP_{3}$ & 2
&
 & $\SP_{2}$, $\SC_{17}(24)$ & 2
\\
196 & $\SC_{2}(3)$, $\SC_{7}(4)$ & 2
&
 & $\SP_{2}$, $\SC_{17}(48)$ & 2
\\
 & $\SC_{2}(3)$, $\SC_{7}(12)$ & 2
&
588 & $\SC_{2}(3)$, $\SP_{3}$, $\SC_{7}(8)$ & 2
\\
225 & $\SC_{3}(8)$, $\SC_{5}(8)$ & 2
&
 & $\SC_{2}(3)$, $\SP_{3}$, $\SC_{7}(24)$ & 2
\\
234 & $\SP_{2}$, $\SC_{3}(4)$, $\PS_{13}(2)$ & 2
&
675 & $\SC_{27}(\sqrt{-3},-1)$, $\SC_{5}(8)$ & 2
\\
242 & $\SP_{2}$, $\SC_{11}(4)$ & 2
&
726 & $\SP_{2}$, $\SP_{3}$, $\SC_{11}(12)$ & 2
\\
 & $\SP_{2}$, $\SC_{11}(20)$ & 2
&
 & $\SP_{2}$, $\SP_{3}$, $\SC_{11}(60)$ & 2
\\
252 & $\SC_{2}(3)$, $\SC_{3}(4)$, $\SP_{7}$ & 2
&
882 & $\SP_{2}$, $\SC_{3}(4)$, $\SC_{7}(4)$ & 2
\\
256 & $\SC_{256}(0)$ & 2
&
 & $\SP_{2}$, $\SC_{3}(4)$, $\SC_{7}(12)$ & 2
\\
270 & $\SP_{2}$, $\SC_{27}(\sqrt{3},1)$, $\SP_{5}$ & 2
&
900 & $\SC_{2}(3)$, $\SC_{3}(8)$, $\SC_{5}(8)$ & 2
\\
294 & $\SP_{2}$, $\SP_{3}$, $\SC_{7}(8)$ & 2
&
1058 & $\SP_{2}$, $\SC_{23}(4)$ & 2
\\
 & $\SP_{2}$, $\SP_{3}$, $\SC_{7}(24)$ & 2
&
 & $\SP_{2}$, $\SC_{23}(44)$ & 2
\\
320 & $\SC_{64}(0)$, $\SP_{5}$ & 2
&
1089 & $\SC_{3}(4)$, $\SC_{11}(4)$ & 2
\\
324 & $\SC_{2}(3)$, $\SC_{81}$ & 2
&
 & $\SC_{3}(4)$, $\SC_{11}(20)$ & 2
\\
350 & $\SP_{2}$, $\SC_{5}(8)$, $\PS_{7}(2)$ & 2
&
1350 & $\SP_{2}$, $\SC_{27}(\sqrt{-3},1)$, $\SC_{5}(3)$ & 2
\\
378 & $\SP_{2}$, $\SC_{27}(\sqrt{-3},-1)$, $\PS_{7}(2)$ & 2
&
 & $\SP_{2}$, $\SC_{27}(\sqrt{-3},1)$, $\SC_{5}(6)$ & 2
\\
396 & $\SC_{2}(3)$, $\SC_{3}(4)$, $\SP_{11}$ & 2
&
 & $\SP_{2}$, $\SC_{27}(\sqrt{-3},1)$, $\SC_{5}(12)$ & 2
\\
441 & $\SC_{3}(4)$, $\SC_{7}(8)$ & 2
&
 & $\SP_{2}$, $\SC_{27}(\sqrt{-3},-1)$, $\SC_{5}(24)$ & 2
\\
 & $\SC_{3}(4)$, $\SC_{7}(24)$ & 2
&
 & $\SP_{2}$, $\SC_{27}(\sqrt{-3},1)$, $\SC_{5}(8)$ & 2
\\
450 & $\SP_{2}$, $\SC_{3}(4)$, $\SC_{5}(3)$ & 2
&
1452 & $\SC_{2}(3)$, $\SP_{3}$, $\SC_{11}(4)$ & 2
\\
 & $\SP_{2}$, $\SC_{3}(4)$, $\SC_{5}(6)$ & 2
&
 & $\SC_{2}(3)$, $\SP_{3}$, $\SC_{11}(20)$ & 2
\\
 & $\SP_{2}$, $\SC_{3}(4)$, $\SC_{5}(8)$ & 3
&
1600 & $\SC_{64}(1)$, $\SC_{5}(8)$ & 2
\\
 & $\SP_{2}$, $\SC_{3}(4)$, $\SC_{5}(12)$ & 2
&
1728 & $\SC_{64}(0)$, $\SC_{27}(\sqrt{-3},1)$ & 2
\\
 & $\SP_{2}$, $\SC_{3}(8)$, $\SC_{5}(24)$ & 2
&
 & $\SC_{64}(1)$, $\SC_{27}(\sqrt{-3},-1)$ & 2
\\
484 & $\SC_{2}(3)$, $\SC_{11}(3)$ & 2
&
1764 & $\SC_{2}(3)$, $\SC_{3}(4)$, $\SC_{7}(4)$ & 2
\\
 & $\SC_{2}(3)$, $\SC_{11}(6)$ & 2
&
 & $\SC_{2}(3)$, $\SC_{3}(4)$, $\SC_{7}(12)$ & 2
\\
 & $\SC_{2}(3)$, $\SC_{11}(15)$ & 2
&
2178 & $\SP_{2}$, $\SC_{3}(4)$, $\SC_{11}(3)$ & 2
\\
 & $\SC_{2}(3)$, $\SC_{11}(30)$ & 2
&
 & $\SP_{2}$, $\SC_{3}(4)$, $\SC_{11}(6)$ & 2
\\
540 & $\SC_{2}(3)$, $\SC_{27}(\sqrt{-3},1)$, $\SP_{5}$ & 2
&
 & $\SP_{2}$, $\SC_{3}(4)$, $\SC_{11}(15)$ & 2
\\
576 & $\SC_{64}(1)$, $\SC_{3}(4)$ & 2
&
 & $\SP_{2}$, $\SC_{3}(4)$, $\SC_{11}(30)$ & 2
\\
 & $\SC_{64}(2)$, $\SC_{3}(8)$ & 2
& & &

\\

\end{tabular}
\vspace{5cm}
\end{center}
\end{table}

\begin{table}
\footnotesize
\caption{Explanation of symbols appearing as local types for $\Q_p$ in Table~\ref{typeslist}, listed with dimension and conductor.}
\begin{center}
\label{legend}
\begin{tabular}{| c | l | c | c |}
Symbol & Definition & Dimension & Conductor \\
\hline
$\ST_p$ & \parbox{3in}{\rule[-.3cm]{0cm}{.3cm}

Steinberg representation of $\GL(2,\Z/p\Z)$. \rule[-.3cm]{0cm}{.3cm} } & $p$ & $p$ \\


$\PS_p(n)$ & \parbox{3in}{\rule[-.3cm]{0cm}{.3cm}

Principal Series rep. of $\GL(2,\Z/p\Z)$ corresponding to the characters $\varepsilon$ and $1$ of $(\Z/p\Z)^*$, where $\varepsilon$ has order $n$. \rule[-.3cm]{0cm}{.3cm} } & $p+1$ & $p$ \\


$\SC_p(n)$ & \parbox{3in}{\rule[-.3cm]{0cm}{.3cm}

Cuspidal rep. of $\GL(2,\Z/p\Z)$ corresponding to a multiplicative character $\theta$ of order $n$ of a quadratic field extension of $\Z/p\Z$. \rule[-.3cm]{0cm}{.3cm} } & $p-1$ & $p^2$ \\

$\PS_4$ & \parbox{3in}{\rule[-.3cm]{0cm}{.3cm}

Principal Series rep. of $\GL(2,\Z/4\Z)$ corresponding to the unique primitive character of conductor 4. \rule[-.3cm]{0cm}{.3cm} } & $6$ & $4$ \\

$\PS_8$ & \parbox{3in}{\rule[-.3cm]{0cm}{.3cm}

Principal Series rep. of $\GL(2,\Z/8\Z)$ corresponding to the unique even primitive character of conductor 8. \rule[-.3cm]{0cm}{.3cm} } & $12$ & $8$ \\

$\PS_9$ & \parbox{3in}{\rule[-.3cm]{0cm}{.3cm}

Principal Series rep. of $\GL(2,\Z/9\Z)$ corresponding to any even primitive character of conductor 9. \rule[-.3cm]{0cm}{.3cm} } & $12$ & $9$ \\

$\SC_8$ & \parbox{3in}{\rule[-.3cm]{0cm}{.3cm} 

Type belonging to a supercusp. rep. of $\GL(2,\Q_2)$ of conductor 8 with even central character. \rule[-.3cm]{0cm}{.3cm} } & $3$ & $8$  \\

$\SC_{64}(n)$ & \parbox{3in}{\rule[-.3cm]{0cm}{.3cm}

Type belonging to a supercusp. rep. of $\GL(2,\Q_2)$ attached to any mult. character $\theta$ of conductor 8 of the unramified extension $\Q_2(\rho)$, where $\rho$ is a primitive 6th root of 1;  the value of $\theta(\rho)$ is a primitive $n$th root of 1. \rule[-.3cm]{0cm}{.3cm} } & $4$ & $64$ \\

$\SC_{256}(n)$ & \parbox{3in}{\rule[-.3cm]{0cm}{.3cm}

Type belonging to a supercusp. rep. of $\GL(2,\Q_2)$ attached to a mult. character $\theta$ of conductor 16 of the unramified extension $\Q_2(\rho)$;  the value of $\theta(\rho)$ is a primitive $n$th root of 1. \rule[-.3cm]{0cm}{.3cm} } & $8$ & $256$ \\

$\SC_{27}(\sqrt{\pm 3},\iota)$ & \parbox{3in}{\rule[-.3cm]{0cm}{.3cm}

Type belonging to a supercusp. rep. $\pi$ of $\GL(2,\Q_3)$ attached to a character $\theta$ of conductor 9 of $\Q_3(\sqrt{\pm 3})^*$, assuming that the central character of $\pi$ has sign $\iota$.  This translates to the condition $\theta(-1)=-\iota$.  \rule[-.3cm]{0cm}{.3cm}  } & $8$ & $27$ \\

$\SC_{81}$ & \parbox{3in}{\rule[-.3cm]{0cm}{.3cm}

Type belonging to any supercusp. rep. of $\pi$ of conductor 81 whose central character is even.  \rule[-.3cm]{0cm}{.3cm} } & $54$ & $81$
\thispagestyle{empty}
\end{tabular}
\end{center}
\end{table}

\section{The field over which $J_1(p^n)$ is semi-stable}

The existence of modular forms with prescribed ramification behavior
has arithmetic consequences for the Jacobians of Shimura curves. For
simplicity we restrict our attention to the case of $K=\Q$;  we
examine the modular Jacobians $J=J_1(p^n)$ for $p\geq 3$ prime.

In~\cite{Krir} an explicit extension $M$ of $\Q_p^{\text{nr}}$ is
constructed over which $J_0(p^n)$ becomes semi-stable.  The result
of this section is a converse to this sort of theorem, whereby we
construct an explicit extension of $\Q_p^{\text{nr}}$ which contains
any other field over which $J_1(p^n)$ becomes semi-stable.

Following the notations of~\cite{Krir}, let $\Omega_i/\Q_p$,
$i=1,2,3$ be the three quadratic extensions of $\Q_p$, with
$\Omega_1/\Q_p$ unramified.  One realization of this scenario is
$\Omega_1=\Q_p(\sqrt{D})$, $\Omega_2=\Q_p(\sqrt{p})$,
$\Omega_3=\Q_p(\sqrt{Dp})$, where $D\in\Z_p^*$ is a quadratic
nonresidue.  For each $i$ let $\gp_i$ be the maximal ideal of
$\Omega_i$ and let $M_i/\Omega_i^{\text{nr}}$ be the class field
with norm subgroup $U_i$ defined by
$$U_i=\begin{cases} \pm (1+\gp_{i}^{\floor{n/2}}) &i=1\\  1+\gp_{i}^{n-1},&i=2,3. \end{cases}$$   
Finally let $M=M_1M_2M_3\Q_p^{\text{nr}}(\zeta_p^n)$.

Let $A_n$ denote the set of two-dimensional Weil-Deligne
representations $\rho_p$ of the Weil group of $\Q_p$ of conductor
dividing $p^n$ and satisfying $\det\rho_p(-1)=1$.  Then $M$ has the
following interpretation:
\begin{equation}
\label{aboutM} \bigcap_{\rho\in A_n}\ker\rho_p\vert_{I_{\Q_p}}\text{
has fixed field precisely $M$}.
\end{equation}
Indeed, any $\rho_p\in A_n$ has one of the following forms:
\begin{enumerate}
\item decomposable as $\eps_1\oplus\eps_2$, where the $\eps_i$ have conductor dividing $p^n$,
\item $\eps\otimes\Sp(2)$, where $\eps$ has conductor dividing $p^n$, or
\item $\Ind_{\Omega_i/\Q_p}\theta$, where $i\in\set{1,2,3}$.
\end{enumerate}
In the last case, the condition that $\rho_p$ has conductor dividing
$p^n$ translates into the condition that $\theta$ has conductor
$\floor{n/2}$ if $i=1$ and $n-1$ if $i=2,3$, as can be determined from the classification in Section~\ref{tracebounds}.  The condition that
$\det\rho_p(-1)=1$ means that $\theta(-1)=1$ if $i=1$, and
$\theta(-1)=(-1)^{(p-1)/2}$ if $i=2,3$.  For a given $i$, the fixed
field of $\rho_p\vert_{I_{\Q_p}}$ for $\rho_p$ arising from such a
character $\theta$ of $\Omega_i^*$ is exactly $M_i$.  The claim in
Eq.~\ref{aboutM} follows.


\begin{Theorem}  $J$ is semi-stable over $M$.  Conversely, if $p^n$ is any odd prime power other than 3,5,7,9,11,13,17,19,27,49, or 121, then $M$ is the minimal extension of $\Q_p^{\text{nr}}$ over
which $J$ becomes semi-stable.
\end{Theorem}

\begin{proof} The variety $J$ is isogenous to $\prod_fJ_f$, where $f$ runs over Galois
orbits of newforms of conductor dividing $p^n$, and where the
$\ell$-adic Galois representation corresponding to $f$ arises from
the $\ell$-adic Tate module of $J_f$.  The abelian variety $J_f$
becomes semi-stable over an extension $L/\Q_p^{\text{nr}}$ if an
only if the local Weil-Deligne representation $\rho_{f,p}$ attached
to $f$ at $p$ becomes unipotent when restricted to
$\Gal(\overline{\Q}_p/L)$;  see~\cite{Groth:SGA}, exp. IX. Since $\rho_{f,p}$ lies in $A_n$,
Eq.~\ref{aboutM} implies that $J_f$ must become semi-stable over the
field $M$.

For the converse statement, suppose $J$ is semi-stable over
$L\supset K^{\text{nr}}$.
 Let $\rho_p\in A_n$
and assume that no twist of $\rho_p$ is unramified.  Let $\pi_p$ be
the admissible representation of $\GL(2,\Q_p)$ corresponding to
$\rho_p$, and let $\tau_p=\tau(\pi_p)$ be its inertial type. Let
$\tau\in\Types(\Q)$ be a global type of weight 2 whose only nontrivial local component is $\tau_p$. The prime powers listed in the theorem are the only ones which appear as conductors in Table~\ref{typeslist}.   Since $p^n$ is not among these, $S(\tau)$
contains a cusp form $f$.  The assumption on $L$ then implies
$W_L\subset\ker\rho_p$.   Since $\rho_p$ was arbitrary, we may apply
Eq.~\ref{aboutM} to conclude that $L\supset M$, whence the theorem.
\end{proof}

\bibliographystyle{amsalpha}
\bibliography{mybibfile}

\end{document}